\documentclass[13pt, reqno]{amsart}
\pdfoutput=1

\usepackage{array}
\usepackage{color}
\usepackage{enumerate}
\usepackage{listings}
\usepackage[colorlinks,linkcolor=black]{hyperref}

\usepackage{amscd}
\usepackage{amsthm}
\usepackage{amsmath}
\usepackage{amsfonts}
\usepackage{amssymb}
\usepackage{mathrsfs}
\usepackage{bm}

\usepackage{euscript}
\usepackage{epsfig}
\usepackage{graphicx}
\usepackage{latexsym}
\usepackage{pdfsync}
\usepackage{wasysym}
\usepackage{stmaryrd}
\usepackage{tikz-cd}
\usepackage{pgfplots}
\usepackage{mathrsfs}
\usetikzlibrary{arrows}

\newcommand{\rnum}[1]{\uppercase\expandafter{\romannumeral #1\relax}}

\newcommand{\foa}{\,\forall\,}

\newcommand{\R}{\mathbb{R}}
\newcommand{\Z}{\mathbb{Z}}

\newcommand{\C}{\mathbb{C}}

\newcommand{\tri}[1]{\left<{}#1\right>}

\newcommand{\pfrac}[2][]{\frac{\partial {}#1}{\partial {}#2}}

\newcommand{\vol}{\mathrm{vol}}

\newcommand{\tr}{\operatorname{tr}}

\newcommand{\diag}{\operatorname{diag}}

\newcommand{\Spec}{\operatorname{Spec}}
\renewcommand{\Re}{\operatorname{Re}}

\newcommand{\FF}{\mathcal{F}}

\newcommand{\HH}{\mathcal{H}}

\newcommand{\PP}{\mathcal{P}}
\newcommand{\II}{\mathcal{I}}

\newcommand{\EE}{\mathcal{E}}

\newcommand{\SSS}{\mathscr{S}}

\newcommand{\FFF}{\mathscr{F}}

\newcommand{\wt}{\widetilde}
\newcommand{\cl}{\overline}

\theoremstyle{theorem}
\newtheorem{thm}{Theorem}[section]
\newtheorem{prop}[thm]{Proposition}

\theoremstyle{lemma}
\newtheorem{lemma}[thm]{Lemma}

\theoremstyle{corollary}
\newtheorem{cor}[thm]{Corollary}

\theoremstyle{definition}
\newtheorem{defi}[thm]{Definition}

\theoremstyle{example}

\theoremstyle{remark}
\newtheorem{rem}[thm]{Remark}

\theoremstyle{proof}

\author{Yong Lin}
\address{Yong Lin: YMSC, Tsinghua University, Beijing 100084, China}
\email{yonglin@tsinghua.edu.cn}

\author{Shi Wan}
\address{Shi Wan: YMSC, Tsinghua University, Beijing 100084, China}
\email{wans21@mails.tsinghua.edu.cn}

\author{Haohang Zhang}
\address{Haohang Zhang: YMSC, Tsinghua University, Beijing 100084, China}
\email{zhanghh22@mails.tsinghua.edu.cn}

\title{Connection Laplacian on discrete tori with converging property}
\date{\today}

\pgfplotsset{compat=1.15}

\begin{document}

\maketitle

\begin{abstract}
	This paper presents a comprehensive analysis of the spectral properties of the connection Laplacian for both real and discrete tori. We introduce novel methods to examine these eigenvalues by employing parallel orthonormal basis in the pullback bundle on universal covering spaces. Our main results reveal that the eigenvalues of the connection Laplacian on a real torus can be expressed in terms of standard Laplacian eigenvalues, with a unique twist encapsulated in the torsion matrix. This connection is further investigated in the context of discrete tori, where we demonstrate similar results.
	
	A significant portion of the paper is dedicated to exploring the convergence properties of a family of discrete tori towards a real torus. We extend previous findings on the spectrum of the standard Laplacian to include the connection Laplacian, revealing that the rescaled eigenvalues of discrete tori converge to those of the real torus. Furthermore, our analysis of the discrete torus occurs within a broader context, where it is not constrained to being a product of cyclic groups. Additionally, we delve into the theta functions associated with these structures, providing a detailed analysis of their behavior and convergence.
	
	The paper culminates in a study of the regularized log-determinant of the connection Laplacian and the converging results of it. We derive formulae for both real and discrete tori, emphasizing their dependence on the spectral zeta function and theta functions.
	
	
\end{abstract}

\section{Introduction}

The study of the eigenvalues of the Laplacian has long been a classical problem in geometry. The discrete Laplacian, along with related topics such as discrete Green's functions and heat kernels, has garnered considerable attention in the discrete geometry community, as evidenced by works like \cite{Y1,C1,CY1,CY2,CY3,Green1}. A notable application of this study is in the analysis of random walks on graphs \cite{Y1}.

Recent advancements have brought the connection Laplacian into focus. This operator, which acts on vector bundles of graphs, draws inspiration from the concept of connections on Riemannian manifolds. There is a growing body of work in this area \cite{Connection1,Connection2,Connection3,Connection4,Connection5,Connection6}. The connection Laplacian is a more general operator, exhibiting a range of intriguing properties not observed in the standard Laplacian \cite{Bundle1,F1}.

A pivotal question in this field concerns the behavior of a family of graphs as they converge to a manifold in a specific manner. Some research has focused on constructing graphs by randomly selecting points on a manifold. As the number of points increases, these graphs form a converging family, with certain geometric quantities approaching those of the manifold in a probabilistic sense \cite{Connection1,Connection6}.

In our paper, we explore a specific manifold, the torus, and a family of graphs known as discrete tori. These discrete tori converge to a real torus in a manner distinct from random point selection. There has been significant progress in understanding the standard Laplacian in this context \cite{Torus1,Converge1,Converge2}. We extend this research by examining the connection Laplacian, calculating its eigenvalues and their log-determinants. Notably, the discrete tori we consider are more general and may not necessarily be products of cyclic groups.

It is also important to acknowledge the work of Finski et al. \cite{Finski1,Finski2}, who have demonstrated convergence results for the eigenvalues of the connection Laplacian on tileable surfaces, including 2-dimensional tori with rational periods. Our study serves both as a specific verification of these findings and as an exploration into higher dimensions and more general settings.

Our first main theorem addresses the eigenvalues of the connection Laplacian on torus. Previous investigations, such as those presented in \cite[Theorem~5.2]{kuwabara1982spectra}, have explored the eigenvalues for a trivial line bundle over a flat torus, employing the Hodge-de Rham theorem and gauge transformations for their analysis. Our approach diverges by utilizing a parallel orthonormal basis within the pullback bundle on the universal covering space, enabling us to extend our findings to more generalized cases.

	\begin{thm}\label{thm:RT}
		For the $d$-dimensional connection Laplacian on a real torus $\R T^n=\R^n/A\Z^n,$ the eigenvalues are given by
		$$
		\lambda_{j,w}=4\pi^2|w|^2,\quad w\in \Gamma^*+\cl w_{j},\quad \cl w_{j}=-(A^T)^{-1}\left(\frac{\omega_{j}^{(k)}}{2\pi}\right)_{k=1}^n,\quad j=1,\cdots,d,
		$$
		where $\Gamma^*=(A^T)^{-1}\Z^n$ is the dual lattice of $\Gamma=A\Z^n,$ and $\omega_{j}^{(k)}$ is the argument such that $e^{\omega_{j}^{(k)}i}$ is the $j$'th eigenvalue of the torsion matrix $\sigma^{v_k},$ as defined in Definition~\ref{def:sigma}.

	\end{thm}
	
	This theorem demonstrates that the eigenvalues of the connection Laplacian on a torus are essentially $d$ groups of standard Laplacian eigenvalues, and the effect of connection is reflected in $w_0$. This extends the findings presented in \cite[Proposition~3.2]{kuwabara1982spectra}, broadening the scope of the original results. Eigenvalues of the standard Laplacian are discussed in Section~\ref{sec:pre}. The following analogous result is observed for the discrete torus, which extends the results in \cite{DT1} to a more general case.
	
	\begin{thm}\label{thm:DT}
		For the $d$-dimensional connection Laplacian on a discrete torus $DT^n=\Z^n/A\Z^n,$ the eigenvalues are given by
		$$
		\lambda_{j,w}=4(\sin^2(w_1\pi)+\cdots+\sin^2(w_n\pi)),\quad w\in (\Gamma^*+\cl w_j)/\Z^n,\quad \cl w_{j}=-(A^T)^{-1}\left(\frac{\omega^{(k)}_{j}}{2\pi}\right)_{k=1}^n
		$$
		for $j=1,\cdots,d,$ where $w_i$ is the $i$'th component of $w.$ Other symbols are consistent with those defined earlier.

	\end{thm}
	
	We also extend the results presented in \cite{Converge2} to more general tori, and additionally apply these findings to the case of the connection Laplacian. In our analysis, we adopt the methodology outlined in \cite{Converge2}, while implementing refinements in certain technical aspects of the proof to extend our results to a broader class of tori. Specifically, we consider real tori of the form $\mathbb{R}^n / A\mathbb{Z}^n$ and discrete tori $\mathbb{Z}^n / M\mathbb{Z}^n$, where $A$ and $M$ are allowed to be arbitrary nonsingular matrices, diverging from the restriction to diagonal matrices imposed in \cite{Converge2}. This generalization enables our analysis to encompass a wider variety of toroidal geometries.
	
	\begin{thm}\label{thm:converge}
		Consider a converging discrete tori family $\frac{1}{u}DT_u^n\rightarrow\R T^n,$ the scaled eigenvalue set $u^2\Spec(L_{DT^n_u})$ converges to $\Spec(L_{\R T^n}),$ the eigenvalue set of real torus. Furthermore, we have
		$$
		\log {\det}^* (L_{DT^n_u})=\mu(DT^n_u)\II_n(0)d+(\dim \ker L_{\R T^n}) \log u^2+\log {\det}^*(L_{\R T^n})+o(1),\quad \text{as }u\rightarrow\infty,
		$$
		where $\mu(DT^n_u)$ counts the number of vertices in $DT^n_u,$ and
		$$
		\II_n(0)=\left(\log 2n-\int_0^\infty (e^{-2nt}(I_0(2t)^n-1)\frac{dt}{t})\right).
		$$
	
	\end{thm}
	
	The converging discrete tori family is briefly introduced in Section~\ref{sec:Converging}, and explained in detail in Definition~\ref{def:converge}. By Theorem~\ref{thm:RT}, it is evident that
	$$\dim\ker L_{\R T^n}=\#\{\cl w_j\mid \cl w_j=0\}.$$
	This finding is in alignment with the result stated in \cite[Corollary~2.6]{Finski1}, which asserts that $\ker L_{\R T^n}$ is isomorphic to the space of flat sections of the vector bundle $\R^d\rightarrow E\rightarrow \R T^n.$	
	
	The structure of the paper is as follows. In Section~\ref{sec:pre}, we introduce some basic notations, tools, and results. The proofs of Theorem~\ref{thm:RT} and Theorem~\ref{thm:DT} are presented in Section~\ref{sec:RT} and Section~\ref{sec:DT}, respectively. In Sections~\ref{sec:heat_kernel} and \ref{sec:heat_kernel_CL}, we discuss results about the heat kernel that are instrumental in deriving results related to the theta function. Finally, in Section~\ref{sec:converge}, we prove Theorem~\ref{thm:converge}. Appendix provides some technical details for the Section~\ref{sec:converge}.

\section{Preliminary}\label{sec:pre}
	\subsection{Laplacian on graphs}
	For a graph $G=(V,E),$ the combinatorial Laplacian $\Delta_G$ is defined as
	$$
	\Delta_G f(x):=\sum_{y\sim x} (f(x)-f(y))=\deg(x)f(x)-\sum_{y\sim x}f(y),\quad x\in V,
	$$
	where $f:V\rightarrow \R$ is a function on the graph. The notation $y \sim x$ denotes that vertices $x$ and $y$ are adjacent, i.e., there is an edge in $E$ connecting them in the graph. The degree $\deg(x)$ of a vertex $x$ is the number of its neighbors.
	
	The inner product of functions $f$ and $g$ is defined to be 
	$$
	\tri{f,g}_G:=\sum_{x\in V} f(x)g(x).
	$$
	
	In this paper, we mainly consider regular graphs. That is, the degree is constant across all vertices. Consequently, $\Delta_G$ is self-adjoint and has non-negative real eigenvalues.
	

%
	
	\subsection{Discretization of torus}
	The standard grid graph is denoted as $\Z^n=(\Z^n,E),$ where we use the notation $\Z^n$ to represent both the graph and the vertex set. The edge set $E$ satisfies that two vertices $x$ and $y$ are adjacent if and only if there exists an $i=1,\cdots,n$ such that $y = x \pm e_i$, where $e_i$ is the $i$'th direction vector with components being zeros except for the $i$'th one being $1$.
	
	A lattice is the set $A\Z^n,$ where $A=[v_1,\cdots,v_n]\in M_n(\R).$ A real torus $\R T^n$ can be expressed as $\R^n/A\Z^n,$ and the discrete torus is the graph $D T^n:=\Z^n/A\Z^n.$ However, for a discrete torus, it is required that $A\in M_n(\Z).$ From elementary geometry, we know that $|\det A|$ represents the volume of the real torus $\mu(\R T^n)$ as well as the size of the discrete torus $\mu(D T^n)$. Figure~\ref{fig:DT2} and Figure~\ref{fig:RT2} are illustrations when $n=2$ and $A=\begin{bmatrix}
		1&3\\
		4&-1
	\end{bmatrix}.$
	
	\begin{figure}[htbp]
	\centering
	\begin{minipage}{.45\textwidth}
		\centering
		\definecolor{qqqqff}{rgb}{0.3,0.3,1}
		\definecolor{qqqqcc}{rgb}{0.5,0.5,0.8}
		\begin{tikzpicture}[line cap=round,line join=round,>=triangle 45,x=1cm,y=1cm,scale=0.7]
			\selectcolormodel{gray}
			\begin{axis}[
				x=1cm,y=1cm,
				axis lines=middle,
				ymajorgrids=true,
				xmajorgrids=true,
				xmin=-1.5,
				xmax=5.5,
				ymin=-2.2,
				ymax=5.2,
				xtick={-3,-2,...,6},
				ytick={-2,-1,...,5},]
				\clip(-1.5,-2.2) rectangle (5.5,5.2);
				\draw [line width=2pt,color=qqqqcc,domain=-1.5:5.5] plot(\x,{(-0--4*\x)/1});
				\draw [line width=2pt,color=qqqqcc,domain=-1.5:5.5] plot(\x,{(-0-1*\x)/3});
				\draw [line width=2pt,color=qqqqcc,domain=-1.5:5.5] plot(\x,{(-13--4*\x)/1});
				\draw [line width=2pt,color=qqqqcc,domain=-1.5:5.5] plot(\x,{(-13--1*\x)/-3});
				\draw [line width=2pt,color=qqqqcc,domain=-1.5:5.5] plot(\x,{(--26-1*\x)/3});
				\draw [line width=2pt,color=qqqqcc,domain=-1.5:5.5] plot(\x,{(--26-4*\x)/-1});
				\draw [line width=2pt,color=qqqqcc,domain=-1.5:5.5] plot(\x,{(--13--4*\x)/1});
				\draw [line width=2pt,color=qqqqcc,domain=-1.5:5.5] plot(\x,{(-13-1*\x)/3});
				\draw [line width=2pt,color=qqqqcc,domain=-1.5:5.5] plot(\x,{(-26-4*\x)/-1});
				\draw [line width=2pt,color=qqqqcc,domain=-1.5:5.5] plot(\x,{(-39--4*\x)/1});
				\draw [line width=2pt] (0,0)-- (1,0);
				\draw [line width=2pt] (1,0)-- (1,1);
				\draw [line width=2pt] (1,1)-- (1,2);
				\draw [line width=2pt] (1,2)-- (1,3);
				\draw [line width=2pt] (1,3)-- (1,4);
				\draw [line width=2pt] (1,3)-- (2,3);
				\draw [line width=2pt] (2,3)-- (3,3);
				\draw [line width=2pt] (3,3)-- (4,3);
				\draw [line width=2pt] (3,3)-- (3,2);
				\draw [line width=2pt] (3,2)-- (3,1);
				\draw [line width=2pt] (3,1)-- (3,0);
				\draw [line width=2pt] (3,0)-- (3,-1);
				\draw [line width=2pt] (3,0)-- (2,0);
				\draw [line width=2pt] (2,0)-- (1,0);
				\draw [line width=2pt] (1,1)-- (2,1);
				\draw [line width=2pt] (2,1)-- (3,1);
				\draw [line width=2pt] (2,0)-- (2,1);
				\draw [line width=2pt] (2,1)-- (2,2);
				\draw [line width=2pt] (2,2)-- (3,2);
				\draw [line width=2pt] (2,2)-- (2,3);
				\draw [line width=2pt] (2,3)-- (2,3.66666);
				\draw [line width=2pt] (2,2)-- (1,2);
				\draw [line width=2pt] (1,2)-- (0.5,2);
				\draw [line width=2pt] (3,1)-- (3.5,1);
				\draw [line width=2pt] (3,2)-- (3.75,2);
				\draw [line width=2pt] (2,0)-- (2,-0.666666);
				\draw [line width=2pt] (1,0)-- (1.,-0.333333);
				\draw [line width=2pt] (3,3)-- (3.,3.333333);
				\draw [line width=2pt] (3,0)-- (3.25,0);
				\draw [line width=2pt] (1,1)-- (0.25,1);
				\draw [line width=2pt] (0.75,3)-- (1,3);
				\begin{scriptsize}
					\draw [fill=qqqqff] (3,-1) circle (2.5pt);
					\draw [fill=qqqqff] (4,3) circle (2.5pt);
					\draw [fill=qqqqff] (0,0) circle (2pt);
					\draw [fill=qqqqff] (1,0) circle (2.5pt);
					\draw [fill=qqqqff] (1,1) circle (2.5pt);
					\draw [fill=qqqqff] (1,2) circle (2.5pt);
					\draw [fill=qqqqff] (1,3) circle (2.5pt);
					\draw [fill=qqqqff] (1,4) circle (2.5pt);
					\draw [fill=qqqqff] (2,3) circle (2.5pt);
					\draw [fill=qqqqff] (3,3) circle (2.5pt);
					\draw [fill=qqqqff] (3,2) circle (2.5pt);
					\draw [fill=qqqqff] (3,1) circle (2.5pt);
					\draw [fill=qqqqff] (3,0) circle (2.5pt);
					\draw [fill=qqqqff] (3,-1) circle (2pt);
					\draw [fill=qqqqff] (2,0) circle (2.5pt);
					\draw [fill=qqqqff] (2,1) circle (2.5pt);
					\draw [fill=qqqqff] (2,2) circle (2.5pt);
				\end{scriptsize}
			\end{axis}
		\end{tikzpicture}
		\caption{Discrete torus}\label{fig:DT2}
	\end{minipage}
	\quad\quad
	\begin{minipage}{.45\textwidth}
		\centering
		\definecolor{zzttqq}{rgb}{0.6,0.2,0}
		\definecolor{qqqqff}{rgb}{0.5,0.5,1}
		\definecolor{qqqqcc}{rgb}{0.5,0.5,1}
			\begin{tikzpicture}[line cap=round,line join=round,>=triangle 45,x=1cm,y=1cm,scale=0.7]
				\selectcolormodel{gray}
				\begin{axis}[
					x=1cm,y=1cm,
					axis lines=middle,
					ymajorgrids=true,
					xmajorgrids=true,
					xmin=-1.5,
					xmax=5.5,
					ymin=-2.2,
					ymax=5.2,
					xtick={-3,-2,...,6},
					ytick={-2,-1,...,5},]
					\clip(-1.5,-2.2) rectangle (5.5,5.2);
					\fill[line width=2pt,color=zzttqq,fill=zzttqq,fill opacity=0.10000000149011612] (1,4) -- (0,0) -- (3,-1) -- (4,3) -- cycle;
					\draw [line width=2pt,color=qqqqcc,domain=-1.5:5.5] plot(\x,{(-0--4*\x)/1});
					\draw [line width=2pt,color=qqqqcc,domain=-1.5:5.5] plot(\x,{(-0-1*\x)/3});
					\draw [line width=2pt,color=qqqqcc,domain=-1.5:5.5] plot(\x,{(-13--4*\x)/1});
					\draw [line width=2pt,color=qqqqcc,domain=-1.5:5.5] plot(\x,{(-13--1*\x)/-3});
					\draw [line width=2pt,color=qqqqcc,domain=-1.5:5.5] plot(\x,{(--26-1*\x)/3});
					\draw [line width=2pt,color=qqqqcc,domain=-1.5:5.5] plot(\x,{(--26-4*\x)/-1});
					\draw [line width=2pt,color=qqqqcc,domain=-1.5:5.5] plot(\x,{(--13--4*\x)/1});
					\draw [line width=2pt,color=qqqqcc,domain=-1.5:5.5] plot(\x,{(-13-1*\x)/3});
					\draw [line width=2pt,color=qqqqcc,domain=-1.5:5.5] plot(\x,{(-26-4*\x)/-1});
					\draw [line width=2pt,color=qqqqcc,domain=-1.5:5.5] plot(\x,{(-39--4*\x)/1});
				\draw [line width=2pt,color=zzttqq] (1,4)-- (0,0);
				\draw [line width=2pt,color=zzttqq] (0,0)-- (3,-1);
				\draw [line width=2pt,color=zzttqq] (3,-1)-- (4,3);
				\draw [line width=2pt,color=zzttqq] (4,3)-- (1,4);
				\begin{scriptsize}
					\draw [fill=qqqqff] (3,-1) circle (2.5pt);
					\draw [fill=qqqqff] (4,3) circle (2.5pt);
					\draw [fill=qqqqff] (0,0) circle (2pt);
					\draw [fill=qqqqff] (1,4) circle (2.5pt);
					\draw [fill=qqqqff] (3,-1) circle (2pt);
					\draw[color=black] (2,1.5) node {$\R T^2$};
				\end{scriptsize}
			\end{axis}
		\end{tikzpicture}
		\caption{Real torus}\label{fig:RT2}
	\end{minipage}
	\end{figure}
	
	To simplify our discussion, let us review the concept of the dual lattice. For a lattice $\Gamma=A\Z^n$, the dual lattice is defined as
	$$\Gamma^*:=\{x\in \R^n\mid \tri{x,y}\in \Z,\foa y\in \Gamma\}.$$
	If the inner product is the standard one, then 
	$$\Gamma^*=(A\Z^n)^*=(A^T)^{-1}\Z^n.$$
	Using the concept of the dual lattice, we can conveniently express the eigenpairs of the standard Laplacian on tori. Here, $\Delta_{\mathbb{R}T^n}$ is induced by $\Delta_{\mathbb{R}^n}=-\sum_{k=1}^n\pfrac[^2]{x_k^2}$.
	
	\begin{prop}\label{prop:RT_eigen}
		For a real torus $\R T^n=\R^n/A\Z^n,$ all the eigenpairs of $\Delta_{\R T^n}$ are given by
		$$
		f_w(x)=e^{2\pi \tri{x,w}i},\quad \lambda_w=4\pi^2|w|^2,\quad \foa w\in \Gamma^*.
		$$
	
	\end{prop}
	
	\begin{proof}
		Firstly it's well defined on the torus because
		$$
		f_w(x+v)=f_w(x)e^{2\pi \tri{v,w}i}=f_w(x),\quad \foa v\in \Gamma.
		$$
		By direct calculation, we have
		$$
		\Delta_{\R T^n} f_w(x)=4\pi^2|w|^2 f_w(x),
		$$
		so the eigenvalues set is $\{4\pi^2|w|^2\mid w\in\Gamma^*\}.$ These are all the eigenvalues because the function set $\{f_w\}$ is dense in $C(\R T^n,\C)$ by Stone-Weierstrass theorem.
		
	\end{proof}

	\begin{rem}
		Although $f_w$ is a complex function, it is important to note that when $w\in \Gamma^*,$ then $-w$ is also in $\Gamma^*$, implying that $f_{-w}=\cl f_w$ is also an eigenfunction with the same eigenvalue. Consequently, both the real and imaginary parts of $f_w$ are eigenfunctions. This is attributable to the fact that the eigenvalue is real. Therefore, $f_w$ and $\cl f_w$ can be replaced by their real and imaginary parts, respectively. Similar properties will be utilized in later discussions.
		
	\end{rem}	
	
	For discrete torus $DT^n=\Z^n/A\Z^n,$ the approach is essentially the same. However, it is important to note that we should have a finite number of eigenvalues now, which is exactly $\mu(DT^n).$
	
	\begin{prop}[{\cite[Lemma~3.1]{Torus1}}]
		All of the eigenpairs of $\Delta_{DT^n}$ are given by
		$$
		f_w(x)=e^{2\pi \tri{x,w}i},\quad \lambda_w=4(\sin^2(w_1\pi)+\cdots+\sin^2(w_n\pi)),\quad \foa w\in \Gamma^*/\Z^n,
		$$
		where $w_k$ is the $k$'th component of $w$.\label{prop:DT_eigen}
	
	\end{prop}

%
	
	\subsection{Converging discrete tori family}\label{sec:Converging}
	To examine the spectral differences between real and discrete tori, it is essential to consider a family of discrete tori converging to the real torus in some sense. One approach is to construct a family of matrices $\{M_u\}_{u=1}^\infty$ such that $\frac{M_u}{u}\rightarrow A$ as $u\rightarrow\infty$, i.e., $\lim\limits_{u\rightarrow\infty}\frac{(M_u)_{ij}}{u}=a_{ij}$ for all $i,j=1,\cdots,n.$ We then define $DT_u^n:=\Z^n/M_u \Z^n,$ representing a series of discrete tori that increase in size. Consequently, we can express $\frac{1}{u}DT_u^n\rightarrow \R T^n=\R^n/A\Z^n.$ The focus then shifts to studying the convergence properties of $u^2\Spec(L_{DT^n_u}),$ the scaled eigenvalue set of $DT^n_u.$
	
	It is straightforward to observe that for standard Laplacian $\Delta_{DT^n_u}$ and $\Delta_{\R T^n},$ the scaled eigenvalue set $u^2\Spec(\Delta_{DT^n_u})$ converges to $\Spec(\Delta_{\R T^n})$ pointwise. This can be considered a simpler version of Proposition~\ref{prop:eigenvalueConverge}. The works in \cite{Converge2, Converge1} have demonstrated the convergence of the log-determinant when both $M_u$ and $A$ are diagonal matrices. In \cite{Finski1, Finski2}, results regarding the connection Laplacian are presented for the case when $M_u = Au$ and $n = 2$. This specific scenario corresponds to the discrete tori being a square partition of the real torus, which is then subdivided increasingly finer. 	
	
	Figure~\ref{fig:DT2u} and Figure~\ref{fig:DT2u-scaled} provide illustrations for the case when $M_u = Au$ and $u = 2$, with the matrix $A$ given by
	$
	A = \begin{bmatrix}
		1 & 3 \\
		4 & -1
	\end{bmatrix},
	$
	consistent with the previous illustration. Thus, Figure~\ref{fig:DT2u} represents a larger discrete torus compared to Figure~\ref{fig:DT2}, and Figure~\ref{fig:DT2u-scaled} is the scaled version of it. As $u \rightarrow \infty$, the subdivision becomes finer, converging to the real torus $\mathbb{R}T^2$ as shown in Figure~\ref{fig:RT2}. In the general case, during the process of $\frac{1}{u}DT^n_u \rightarrow \mathbb{R}T^n_u$, some perturbation of $M_u$ is allowed for the discrete torus.

	\begin{figure}[htbp]
		\centering
		\begin{minipage}{.45\textwidth}
			\centering
			\definecolor{uuuuuu}{rgb}{0.26666666666666666,0.26666666666666666,0.26666666666666666}
			\definecolor{qqqqff}{rgb}{0.5,0.5,1}
			\definecolor{ududff}{rgb}{0.30196078431372547,0.30196078431372547,1}
			\definecolor{xdxdff}{rgb}{0.49019607843137253,0.49019607843137253,1}
			\begin{tikzpicture}[line cap=round,line join=round,>=triangle 45,x=1cm,y=1cm,scale=0.4]
				\selectcolormodel{gray}
				\begin{axis}[
					x=1cm,y=1cm,
					axis lines=middle,
					ymajorgrids=true,
					xmajorgrids=true,
					xmin=-1.5,
					xmax=9.5,
					ymin=-3.5,
					ymax=9.5,
					xtick={-1,0,...,9},
					ytick={-3,-2,...,9},]
					\clip(-10,-3.5) rectangle (9.5,9.5);
					\draw [line width=2pt,color=qqqqff,domain=-1.5:9.5] plot(\x,{(-0--8*\x)/2});
					\draw [line width=2pt,color=qqqqff,domain=-1.5:9.5] plot(\x,{(-0-2*\x)/6});
					\draw [line width=2pt,color=qqqqff,domain=-1.5:9.5] plot(\x,{(-52--8*\x)/2});
					\draw [line width=2pt,color=qqqqff,domain=-1.5:9.5] plot(\x,{(-52--2*\x)/-6});
					\draw [line width=2pt] (0,0)-- (6.5,0);
					\draw [line width=2pt] (2,8)-- (2.0,-0.666);
					\draw [line width=2pt] (3,7.6666)-- (3,-1);
					\draw [line width=2pt] (4,-1.3333)-- (4.0,7.3333);
					\draw [line width=2pt] (5,-1.6666)-- (5,7);
					\draw [line width=2pt] (6,6.6666)-- (6,-2);
					\draw [line width=2pt] (7,6.3333)-- (7,2);
					\draw [line width=2pt] (5,7)-- (1.75,7.0);
					\draw [line width=2pt] (1.5,6)-- (8,6);
					\draw [line width=2pt] (7.75,5.0)-- (1.25,5.0);
					\draw [line width=2pt] (1,4)-- (7.5,4);
					\draw [line width=2pt] (7.25,3.0)-- (0.75,3);
					\draw [line width=2pt] (0.5,2)-- (7,2);
					\draw [line width=2pt] (6.75,1)-- (0.25,1);
					\draw [line width=2pt] (1.,-0.3333)-- (1,4);
					\draw [line width=2pt] (6.25,-1)-- (3,-1);
				\end{axis}
			\end{tikzpicture}
			\caption{$DT^2_u$}\label{fig:DT2u}
		\end{minipage}
		\quad\quad
		\begin{minipage}{.45\textwidth}
			\centering
			\definecolor{xdxdff}{rgb}{0.5,0.5,1}
			\definecolor{qqqqff}{rgb}{0.5,0.5,1}
			\definecolor{ududff}{rgb}{0.3,0.3,1}
			\definecolor{uuuuuu}{rgb}{0.26666666666666666,0.26666666666666666,0.26666666666666666}
			\begin{tikzpicture}[line cap=round,line join=round,>=triangle 45,x=1cm,y=1cm,scale=0.8]
				\selectcolormodel{gray}
				\begin{axis}[
					x=1cm,y=1cm,
					axis lines=middle,
					ymajorgrids=true,
					xmajorgrids=true,
					xmin=-1.5,
					xmax=5.5,
					ymin=-1.5,
					ymax=5.5,
					xtick={-1,0,...,5},
					ytick={-1,0,...,5},]
					\clip(-1.5,-1.5) rectangle (5.5,5.5);
					\draw [line width=2pt,color=qqqqff,domain=-1.5:5.5] plot(\x,{(-0--4*\x)/1});
					\draw [line width=2pt,color=qqqqff,domain=-1.5:5.5] plot(\x,{(-0-1*\x)/3});
					\draw [line width=2pt,color=qqqqff,domain=-1.5:5.5] plot(\x,{(-13--4*\x)/1});
					\draw [line width=2pt,color=qqqqff,domain=-1.5:5.5] plot(\x,{(-13--1*\x)/-3});
					\draw [line width=2pt] (0,0)-- (3.25,0);
					\draw [line width=2pt] (3,-1)-- (2.999,3.333);
					\draw [line width=2pt] (3.5,3.167)-- (3.5,1);
					\draw [line width=2pt] (2.5,3.5)-- (2.5,-0.833);
					\draw [line width=2pt] (2.0,-0.666)-- (1.999,3.666);
					\draw [line width=2pt] (1.5,3.833)-- (1.5,-0.5);
					\draw [line width=2pt] (1.0,-0.3333)-- (1,4);
					\draw [line width=2pt] (0.5,2)-- (0.5,-0.167);
					\draw [line width=2pt] (4,3)-- (0.75,3.0);
					\draw [line width=2pt] (2.5,3.5)-- (0.875,3.5);
					\draw [line width=2pt] (0.625,2.5)-- (3.875,2.5);
					\draw [line width=2pt] (3.75,1.999)-- (0.5,2);
					\draw [line width=2pt] (0.375,1.5)-- (3.625,1.5);
					\draw [line width=2pt] (3.5,1)-- (0.25,1.0);
					\draw [line width=2pt] (0.125,0.5)-- (3.375,0.5);
					\draw [line width=2pt] (3.125,-0.5)-- (1.5,-0.5);
				\end{axis}
			\end{tikzpicture}
			\caption{$\frac{1}{u}DT^2_u$}\label{fig:DT2u-scaled}
		\end{minipage}
	\end{figure}

\section{Connection Laplacian on real torus}\label{sec:RT}
\subsection{Connection of vector bundles on real torus}
	To analyze the spectrum of the connection Laplacian, we begin with the continuous case. Consider a $d$-dimensional vector bundle
	$$
	\R^d\cong F\rightarrow E\rightarrow M=\R T^n,
	$$
	equipped with a smooth positive definite quadratic form $h^F_x,$ serving as a metric for each fiber $F_x$. The inner product between two smooth vector fields (or sections) $U_1, U_2 \in \Gamma(\mathbb{R}T^n, E)$ is defined as
	$$
	\tri{U_1,U_2}=\int_{\R T^n} h^F_x(U_1(x),U_2(x))d\vol_g.
	$$
	
	The connection on the vector bundle is an $\R$-linear map $\nabla:\Gamma(\mathbb\R T^n,E)\rightarrow\Gamma(\R T^n,E\otimes T^*M),$ such that the product rule
	$$
	\nabla(fV)=df\otimes V+f\nabla V
	$$
	holds for all smooth functions $f$ on $\R T^n$ and all smooth vector fields $V$ of $E.$ We may also express the connection as
	$$
	(\nabla V)(X)=\nabla_X V\in F_x,\quad \foa x\in \R T^n,\quad X\in T_xM,\quad V\in \Gamma(\R T^n,E), 
	$$
	and then $\nabla$ is tensorial in $X$ and acts as a derivation on $V.$
	
	Let $\gamma:[0,1]\rightarrow \R T^n$ be a smooth path in $\R T^n.$ A vector field $V$ along $\gamma$ is said to be parallel if 
	$$\nabla_{\dot\gamma(t)}V\equiv 0,\quad \foa t\in [0,1].$$
	By solving the associated ordinary differential equation, we can define the parallel transport 
	$$\PP_{\gamma(t),\gamma(0)}^{\gamma}:F_{\gamma(0)}\rightarrow F_{\gamma(t)},$$ 
	such that $\PP_{\gamma(t),\gamma(0)}^{\gamma}v$ is parallel along $\gamma$ with the initial vector being $v$ at $\gamma(0)$. This parallel transport is a linear isomorphism.
	
	To investigate further, we need to consider additional properties of the vector bundle.
	\begin{defi}
		The connection on a vector bundle $F\rightarrow E\rightarrow M$ can possess the following properties:
		
		\begin{itemize}
			\item The connection is \emph{flat} if the parallel transport along any loop in $M$ depends only on the homotopy class of the loop. It is equivalent to stating that the curvature of the connection is zero. For more details about the curvature of connections, see \cite{BundleConnection}.

			\item The connection is \emph{unitary} if it preserves the metric on each fiber, i.e., the parallel transport along any path in $M$ is a unitary transformation with respect to the quadratic form $h^F$.
		\end{itemize}
		
		A \emph{flat and unitary vector bundle} over a smooth manifold $M$ is a vector bundle $F \rightarrow E \rightarrow M$ equipped with a connection that is both flat and unitary.
		
	\end{defi}	
	
	We always consider the flat and unitary vector bundle, and take the flat metric over $\R T^n$ in later discussion.

\subsection{Connection Laplacian on real torus}
	The connection Laplacian on the vector bundle is defined as 
	$$L_{\R T^n}=(\nabla^*)\nabla,$$
	where $\nabla^*$ is the formal adjoint of $\nabla$ under the $L^2$ inner product. Owing to our preceding assumptions, we can derive that
	$$L_{\R T^n} V=-\sum_{i=1}^n\nabla_{e_i}\nabla_{e_i}V\in F_x,$$
	for any orthonormal basis $\{e_i\}_{i=1}^n$ of $T_xM.$ A general expression within a local coordinate system is detailed in \cite[Proposition~2.2]{kuwabara1982spectra}.
	
	At any point $x \in \mathbb{R}T^n$, we can assign a vector $E_0(x) \in F_x$. For any $y\in U_x,$ where $U_x$ is a sufficiently small neighborhood, we can select any smooth path $\gamma_{xy}\subset U_x$ connecting $x$ and $y.$ We then define $E_0(y)$ as the parallel transport of $E_0(x)$ along $\gamma_{xy}.$ This is well-defined due to the flatness of the connection and $\pi_1(U_x)$ is trivial. Thus we can always take an arbitrary orthonormal basis of $F_x$ and perform parallel transport locally in $U_x$. Since the vector bundle is unitary, the parallel vector fields remain orthonormal basis for each fiber. Normally, global parallel transportation is not feasible due to the non-simply connected nature of the torus.
	
	However, by pulling back the vector bundle to the universal covering space $\R^n \xrightarrow{\pi} \mathbb{R}T^n$, which is simply connected, we can find parallel vector fields $E_1, \cdots, E_d$ over $\R^n.$ These sections form an orthonormal basis for each fiber of the pullback bundle. Consequently, every smooth vector field $\overline{V}$ on $\mathbb{R}^n$ can be decomposed into $\overline{V} = f_1 E_1 + \cdots + f_d E_d$, where each $f_i: \mathbb{R}^n \rightarrow \mathbb{R}$ is a smooth real function.
		
	\begin{lemma}
		For any smooth vector field $V\in \Gamma(\R T^n,E),$ $L_{\R T^n}V=\lambda V$ if and only if $\Delta_{\R^n} f_i=\lambda f_i,$ $\foa i=1,\cdots,d.$ Here $f_i:\R^n\rightarrow \R$ satisfies that $\cl V=\sum_{i=1}^d f_iE_i$ is the lift of $V.$ \label{lem:RTeq}
		
	\end{lemma}
	
	\begin{proof}
		Since $\{E_i\}_{i=1}^d$ form an orthonormal basis for each fiber pointwise, the decomposition is unique, meaning that $\{f_i\}$ is unique. Given that the basis is parallel, we have
		$$
		\begin{aligned}
			L_{\R T^n} V&=L_{\R^n} \cl V&\text{pullback property}\\
			&=L_{\R^n} (f_1E_1+\cdots+f_dE_d)\\
			&=-\sum_{i=1}^n\sum_{j=1}^d \nabla_{\partial_i}\nabla_{\partial_i}(f_jE_j)\\
			&=\sum_{i=1}^d(\Delta_{\R^n} f_i) E_i&\nabla_{\partial_i} E_j=0
		\end{aligned}
		$$
		As $\{E_i\}_{i=1}^d$ form a basis pointwise, it follows that$L_{\R T^n} V=\lambda V$ if and only if $\Delta_{\R^n} f_i=\lambda f_i,$ $i=1,\cdots,d.$
	\end{proof}

	Now the problem is reduced to considering the eigenpairs of the standard Laplacian. We only need to ensure that $\cl V=\sum_{i=1}^d f_iE_i$ is a lift of $V:\R T^n\rightarrow \R^d.$ Specifically, this requires that $\cl V(x)=\cl V(x+v),$ $\foa v\in \Gamma.$ 
	
	\begin{defi}\label{def:sigma}
		Given that $E_i(x+v)\in \cl F_{x+v}=\cl F_{x}=\pi^* F_{[x]},$ we define the \textit{torsion matrix} $\sigma^v(x)$ for each $x\in \R^n$ and $v\in \Gamma,$ such that
		\begin{align}
			E_i(x+v)=\sum_{j=1}^d \sigma_{ji}^v(x)E_j(x),\quad \foa i=1,\cdots,d.\label{eq:basis}
		\end{align}
		
	\end{defi}
	
	It's evident that $\{\sigma_{ji}^v(x)\}_{j,i=1}^d$ is an orthogonal matrix, as both sides of $(\ref{eq:basis})$ represent orthonormal basis. To simplify notation, let $\EE=(E_1,\cdots,E_d),$ and then $(\ref{eq:basis})$ can be expressed as
	$$
	\EE(x+v)=\EE(x)\sigma^v(x).
	$$
	
	\begin{prop}
		For a fixed $ v \in \Gamma$, the torsion matrix $\sigma^v(x)$ is constant for all $ x \in \mathbb{R}^d$, i.e., $\sigma^v(x) \equiv \sigma^v$. \label{prop:const}
	\end{prop}
	
	\begin{proof}
		For any $x, y \in \R^n $, denote the parallel transport from $ x $ to $y$ by $\mathcal{P}_{yx}: F_x \cong F_y.$ It is independent with the path from $x$ to $y$ because $\R^n$ is simply connected. Then $ E_i(y) = \mathcal{P}_{yx}E_i(x)$, and we can write $\mathcal{E}(y) = \mathcal{P}_{yx}\mathcal{E}(x)$. Since the connection on $\mathbb{R}^n$ is the lift of the connection on $\mathbb{R}T^n$, we have $\mathcal{P}_{yx} = \mathcal{P}_{y+v,x+v}$. Therefore, we obtain:
		$$
		\begin{aligned}
			\mathcal{E}(y+v) &= \mathcal{P}_{y+v,x+v}\mathcal{E}(x+v) = \mathcal{P}_{y,x}\mathcal{E}(x)\sigma^v(x) = \mathcal{E}(y)\sigma^v(x).
		\end{aligned}
		$$
		Since we also have $\mathcal{E}(y+v) = \mathcal{E}(y)\sigma^v(y)$, we can conclude that $\sigma^v(x)$ is constant.
	\end{proof}
	
	Let $ \mathcal{F} = (f_1, \cdots, f_d)^T $. Then the condition that $ \overline{V}(x+v) = \overline{V}(x) $ can be expressed as
	$$
	\overline{V}(x+v) = \mathcal{E}(x+v)\mathcal{F}(x+v) = \mathcal{E}(x)\sigma^v\mathcal{F}(x+v) = \overline{V}(x),
	$$
	and it holds if and only if 
	\begin{align}
		\mathcal{F}(x) = \sigma^v\mathcal{F}(x+v). \label{eq:periodic_condition}
	\end{align}
	This can be referred to as the \textit{periodic condition}.	
		
	\begin{lemma}\label{lem:commutativity}
		Due to the flatness of the bundle, the torsion matrices satisfy
		$$
		\sigma^v\sigma^u = \sigma^{v+u} = \sigma^{u}\sigma^v.
		$$
		
	\end{lemma}
	
	\begin{proof}
		Take arbitrary point $ x \in \mathbb{R}^n $. It suffices to examine the relationships between $\mathcal{E}(x)$, $\mathcal{E}(x + v)$, $\mathcal{E}(x + u)$, and $\mathcal{E}(x + v + u)$. Via parallel transport, we have:
		$$
		\mathcal{E}(x)\sigma^u\sigma^v = \mathcal{E}(x + u)\sigma^v = \mathcal{E}(x + v + u) = \mathcal{E}(x + v)\sigma^u = \mathcal{E}(x)\sigma^v\sigma^u.
		$$
		Considering also that $\mathcal{E}(x + u + v) = \mathcal{E}(x)\sigma^{u+v}$, we arrive at the conclusion.
	\end{proof}
	
	By Lemma~\ref{lem:commutativity}, it is sufficient to verify (\ref{eq:periodic_condition}) for all $x\in \R^n$ when $ v = v_k $ for each $ k = 1, \cdots, n $, where $ v_k $ is the column vector of the matrix $ A $, and $ \Gamma = A\mathbb{Z}^n $. After that, $\foa v\in \Gamma,$ $v=c_1v_1+\cdots+c_nv_n,$ and then \eqref{eq:periodic_condition} holds for all $x\in \R^n$ and $v\in \Gamma$ as
	$$
	\begin{aligned}
		\FF(x)=\sigma^{v_1}\FF(x+v_1)=(\sigma^{v_1})^2\FF(x+2v_1)=\cdots=\sigma^{c_1v_1}\FF(x+c_1v_1)=\cdots=\sigma^v\FF(x+v).
	\end{aligned}
	$$

	\subsection{Eigenvalue of connection Laplacian on real torus}\label{sec:LambdaRT}
	Now, we proceed to calculate the eigenvalue of the connection Laplacian on the real torus and prove the theorem.
	
	\begin{proof}[Proof of Theorem~\ref{thm:RT}]
		Invoking Lemma~\ref{lem:commutativity} again, we have $\sigma^{v_i}\sigma^{v_j} = \sigma^{v_j}\sigma^{v_i}$ for all $i, j = 1, \cdots, n$. A basic result from linear algebra tells us that we can diagonalize them simultaneously. Hence, there exists $P \in U(d)$ such that $P^{-1}\sigma^{v_k}P = \Lambda_k$, where $\Lambda_k$ is diagonal, for all $k = 1, \cdots, n$.
		
		The periodic condition (\ref{eq:periodic_condition}) can now be rewritten as 
		$$
		P^{-1}\mathcal{F}(x) = P^{-1}\sigma^{v_k}P P^{-1}\mathcal{F}(x + v_k).
		$$
		If we let $\widetilde{\mathcal{F}} = P^{-1}\mathcal{F}$, then $\widetilde{\mathcal{F}}(x) = \Lambda_k\widetilde{\mathcal{F}}(x + v_k)$, i.e.,
		\begin{align}
			\widetilde{f}_j(x) = \mu^{(k)}_{j}\widetilde{f}_j(x + v_k), \quad \foa k = 1, \cdots, n, \; j = 1, \cdots, d. \label{eq:diagonal_expression}
		\end{align}
		
		Here, $\mu^{(k)}_{j}$ is the $j$'th eigenvalue of $\sigma^{v_k} \in O(d)$, so $\mu^{(k)}_{j} = e^{\omega^{(k)}_{j} i}$, with $\omega^{(k)}_{j} \in (-\pi, \pi]$. Similar to Proposition~\ref{prop:RT_eigen}, we can easily show that $\widetilde{f}_{j,w}(x) = e^{2\pi i\langle x, w \rangle}$ is the eigenfunction satisfying (\ref{eq:diagonal_expression}) with eigenvalue $\lambda_{j,w} = 4\pi^2|w|^2$, for all $w \in \Gamma^* + \cl w_{j}$, where
		$$
		\cl w_{j} = -(A^T)^{-1}\left(\frac{\omega^{(1)}_{j}}{2\pi}, \cdots, \frac{\omega^{(n)}_{j}}{2\pi}\right)^T.
		$$
		
		We need to verify condition (\ref{eq:diagonal_expression}). This holds because
		$$
		\widetilde{f}_{j,w}(x + v_k) = e^{2\pi i\langle v_k, w \rangle}\widetilde{f}_{j,w}(x), \quad e^{2\pi i\langle v_k, w \rangle} = e^{2\pi i\langle v_k, \cl w_{j} \rangle} = e^{-2\pi i \cdot \frac{\omega^{(k)}_{j}}{2\pi}} = \left(\mu^{(k)}_{j}\right)^{-1}.
		$$
		
		Now, we can take the eigenbasis $\{\widetilde{\mathcal{F}}_{j,w} = (0, \cdots, 0, \widetilde{f}_{j,w}, 0, \cdots, 0) \mid w \in \Gamma^* + \cl w_{j}\}_{j=1}^d$, and $\widetilde{\mathcal{F}}_{j,w}$ corresponds to the eigenvalue $\lambda_{j,w} = 4\pi^2|w|^2$. By $\mathcal{F} = P\widetilde{\mathcal{F}}$, we obtain the eigen-system of the original problem.
	\end{proof}

	\begin{rem}
		In our analysis, we employ complex-valued functions for convenience. Since the connection Laplacian is self-adjoint, its eigenvalues are real. Thus, we can consider the real parts and the imaginary parts of the eigenfunctions separately. 
		
		An interesting observation is that for $w\neq 0,$ both the real and imaginary parts of $\mathcal{F}_{j,w}$ are non-zero and independent. Consequently, the eigenvalue $\lambda_{j,w}$ must be doubly duplicated. This leads to the following corollary.
		
	\end{rem}

	\begin{cor}
		Every non-zero eigenvalue of the connection Laplacian on a real torus has an even multiplicity.
		
	\end{cor}
	
\section{Connection Laplacian on discrete torus}\label{sec:DT}
	\subsection{Vector Bundles and Connection on Discrete Torus}
	To study the connection Laplacian on the discrete torus, it is first necessary to define the connection and vector bundle on it.
	
	\begin{defi}
		For a finite graph $G=(V,E)$, we define the $F$-bundle as
		$$
		E_G = \bigoplus_{v \in V} F_v \cong F^{|G|}.
		$$
	\end{defi}
	
	The distinctive characteristics of each bundle are reflected in the connection defined on it.
	
	\begin{defi}
		We say that $\Phi$ is a connection on the bundle if, for each pair of adjacent vertices $u \sim v$, there exists a linear isomorphism $\phi_{uv} : F_v \cong F_u$ with $\phi_{vu} = \phi_{uv}^{-1}$. The connection Laplacian is then defined as
		$$
		L_G : \Gamma(V, E_G) \rightarrow \Gamma(V, E_G),
		$$
		$$
		L_G H(u) = \sum_{v \sim u} H(u) - \phi_{uv}H(v).
		$$
	\end{defi}
	
	The concept of the connection Laplacian on a graph originates from the connection on vector bundles over manifolds, which relates vectors on adjacent fibers by parallel transportation. Therefore, it is natural to define the connection on a graph analogously. 
	
	Suppose we already have a connection on the real torus $\mathbb{R}T^n = \mathbb{R}^n / A\mathbb{Z}^n$ with $A$ an integer matrix, we can pullback the connection on $\mathbb{R}^n$, then induce it to $\frac{1}{u}\mathbb{Z}^n$ for some $u\in \Z_+$ by defining $\phi_{xy}=\PP_{xy}$, the parallel transport from $x$ to $y$ along arbitrary path. Finally the connection is induced to $\frac{1}{u}DT^n_u = \frac{1}{u}(\mathbb{Z}^n / (Au)\mathbb{Z}^n) = (\frac{1}{u}\mathbb{Z}^n) / A\mathbb{Z}^n$ by canonical map. This approach aligns with the idea of connections on tileable surfaces, as discussed in \cite{Finski1,Finski2}, but extends to arbitrary dimensions.
	
	For any connection $\Phi$ on discrete tori $DT^n = \mathbb{Z}^n / A\mathbb{Z}^n$, in order to study the spectrum of the connection Laplacian on it, we need to consider additional conditions for the connection. Recall that in the context of the real torus, the connection we are interested in is characterized by being both unitary and flat. In the discrete setting, the notion of unitarity implies that $\phi_{uv} \in O(d)$ for all adjacent vertices $u, v$. To further understand flatness, we introduce the following definition:
	
	\begin{defi}
		Let $(G, \Phi)$ be a connection graph. For any path $P: x_0 \sim x_1 \sim \cdots \sim x_n$, we define the signature of $P$ as
		$$ 
		\phi_{P} = \phi_{x_{0},x_{1}} \cdots \phi_{x_{i-1},x_{i}} \cdots \phi_{x_{n-1},x_n},
		$$ 
		which represents the parallel transport from $x_n$ to $x_0$ along $P$. If the signature of every cycle in $G$ is the identity map, we call $(G, \Phi)$ a consistent graph. In this case, $\phi_P$ can be denoted as $\phi_{x_0,x_n}$, as it is independent of the choice of path $P$ from $x_0$ to $x_n$.
	\end{defi}
	
	Flatness in the context of vector bundles in continuous case implies local triviality of the connection, although it may not be globally trivial. In the discrete setting, this suggests that the pullback connection on $\mathbb{Z}^n$ is consistent, but there is no requirement for global consistency on $DT^n$. Similar to the continuous case, in our study of discrete tori, we always consider vector bundles that are both unitary and flat.
	
	\begin{prop}
		Suppose $\R^d\cong F\rightarrow E\rightarrow G$ is a flat vector bundles with connection $\Phi$. Fixed any vertex $x_0\in G,$ the vector field $V(y):=\phi_{yx_0}V(x_0)$ is parallel for arbitrary $V(x_0)\in F_{x_0},$ i.e., $L_{G}V\equiv 0.$
		
	\end{prop}
	
	\begin{proof}
		By the definition of $\phi_{yx_0}$ and the flatness of the bundle, for any $z \in \mathbb{Z}^d$, we have $\phi_{yx_0} = \phi_{yz}\phi_{zx_0}$. Consequently, $V(y) = \phi_{yz}V(z)$. Noting that $\phi_{xy} = \phi_{yx}^{-1}$, we obtain:
		$$
		L_GV(x) = \sum_{y \sim x} V(x) - \phi_{xy}V(y) = \sum_{y \sim x} V(x) - \phi_{xy}\phi_{yx}V(x) = 0.
		$$
	\end{proof}

	Now, choose $\{E_i(x_0)\}_{i=1}^d$ to be an orthonormal basis. Then, by parallel transportation and the unitarity of the bundle, $\{E_i(x)\}_{i=1}^d$ forms an orthonormal basis for each fiber $\cl F_x$. For any arbitrary vector field $\overline{V}$ on $\mathbb{Z}^d$, we can decompose it as $\overline{V}(x) = \sum_{i=1}^d f_i(x)E_i(x)$.
	
	\begin{lemma}
		For any vector field $V \in \Gamma(DT^n, E_{DT^n})$, $L_{DT^n}V = \lambda V$ if and only if $\Delta_{\mathbb{Z}^n}f_i = \lambda f_i$ for all $i = 1, \cdots, d$. Here, $f_i : \mathbb{Z}^n \rightarrow \mathbb{R}$ are such that $\overline{V} = \sum_{i=1}^d f_iE_i$ is the lift of $V$.
	\end{lemma}
	
	\begin{proof}
		The proof of this lemma is similar to that of Lemma~\ref{lem:RTeq}, with only minor differences:
		$$
		\begin{aligned}
			L_{DT^n}V([x]) &= L_{\Z^n} \overline{V}(x) &\text{pullback property}\\
			&= \sum_{i=1}^d\sum_{y \sim x} f_i(x)E_i(x) - \phi_{xy}f_i(y)E_i(y) \\
			&= \sum_{i=1}^d\sum_{y \sim x} (f_i(x) - f_i(y))E_i(x) \\
			&= \sum_{i=1}^d(\Delta_{\mathbb{Z}^n} f_i(x)) E_i(x).
		\end{aligned}
		$$
		Since $\{E_i\}_{i=1}^d$ forms a basis pointwise, it follows that $L_{DT^n}V = \lambda V$ if and only if $\Delta_{\mathbb{Z}^n} f_i = \lambda f_i$ for each $i = 1, \cdots, d$.
	\end{proof}
	
	Following a similar approach to the continuous case, we define the connection matrix $\sigma^v(x)$ such that
	$$
	\mathcal{E}(x + v) = \mathcal{E}(x)\sigma^v(x), \quad \text{where } \mathcal{E} = (E_1, \cdots, E_d).
	$$
	
	The parallel transport from $x$ to $y$ is given by $\phi_{yx}$ in discrete version. Accordingly, we have the following lemma:
	
	\begin{lemma}
		The matrix $\sigma^v(x)$ is constant, denoted as $\sigma^v$, and for all $u, v \in \Gamma$, the matrices satisfy the commutative property: $\sigma^v\sigma^u = \sigma^{u+v} = \sigma^u\sigma^v$.
	\end{lemma}
	
	The proof of this lemma follows the same logic as Proposition~\ref{prop:const} and Lemma~\ref{lem:commutativity}, relying on the properties of parallel transport and the flatness of the connection in the discrete setting.
	
	\subsection{Eigenvalue of connection Laplacian on discrete torus}\label{sec:LambdaDT}
	
	The proof of Theorem~\ref{thm:DT} follows a similar approach to Theorem~\ref{thm:RT}, with adjustments for the discrete setting. The key differences lie in the finiteness of the eigenvalue spectrum and the specific expression of eigenvalues in $DT^n$.
	
	\begin{proof}[Proof of Theorem~\ref{thm:DT}]
		Consider $\mathcal{F} = (f_1, \cdots, f_d)^T$. The periodic condition remains
		$$
		\mathcal{F}(x) = \sigma^{v_k}\mathcal{F}(x + v_k), \quad \foa k = 1, \cdots, n.
		$$
		Through simultaneous diagonalization of $\{\sigma^{v_k}\}_{k=1}^n$, for $\widetilde{\mathcal{F}} = P^{-1}\mathcal{F}$, we obtain
		$$
		\widetilde{\mathcal{F}}(x) = \Lambda_k\mathcal{F}(x + v_k), \quad \foa k = 1, \cdots, n.
		$$
		
		Similar to Proposition~\ref{prop:DT_eigen}, the eigenfunction can be expressed as
		$$
		\widetilde{f}_{j,w}(x) = e^{2\pi i\langle x, w \rangle}, \quad w \in (\Gamma^*+ \cl w_{j})/\mathbb{Z}^n,
		$$
		where $\cl w_{j} = -(A^T)^{-1}\left(\frac{\omega^{(k)}_{j}}{2\pi}\right)_{k=1}^n$, and $\mu^{(k)}_j=e^{\omega^{(k)}_ji}$ is the $j$'th eigenvalue of $\sigma^{v_k}.$ The corresponding eigenvalue of $\wt f_{j,w}$ is
		$$
		\lambda_{j,w} = 2n - 2\cos(2\pi w_1) - \cdots - 2\cos(2\pi w_n) = \sum_{k=1}^n \sin^2(\pi w_k),
		$$
		with $w_k$ being the $k$'th component of $w$.
		
		Consequently, we obtain exactly $$\mu(\Gamma^*/\Z^n)=\mu((A^T)^{-1}(\Z^n/(A^T)\Z^n))=|\det A^T| = |\det A| = \mu(DT^n) = N$$ different eigenfunctions $\widetilde{f}_{j,w}$, leading to a total of $dN$ independent vector-valued eigenfunctions $\mathcal{F} = P\widetilde{\mathcal{F}}$. This accounts for all eigenpairs of the connection Laplacian on $DT^n$.
	\end{proof}
	
	\begin{rem}
		In the continuous case, both the real and imaginary parts of $\mathcal{F}$ are non-zero whenever $w \neq 0$. However, in the discrete setting, this is not always the case due to the discrete nature of $\mathcal{F}$ on $DT^n$, leading to instances where not every eigenvalue has an even multiplicity. Nevertheless, we can still assert the following corollary.
	\end{rem}
	
	\begin{cor}
		Suppose $\sigma^v \in SO(2d)$ for all $v \in \Gamma$, then each eigenvalue of the connection Laplacian $L_{DT^n}$ on the discrete torus has an even multiplicity.
		
	\end{cor}
	
	\begin{proof}
		This result arises from the properties of commutative orthogonal groups. We can ensure that each $\Lambda_k$ has a similar structure:
		$$
		\Lambda_k = \diag\{\eta_{k,1}, \overline{\eta}_{k,1}, \cdots, \eta_{k,m_k}, \overline{\eta}_{k,m_k}, 1, \cdots, 1, -1, \cdots, -1\}, \quad |\eta_{k,i}| = 1,
		$$
		where $\Lambda_k \in SO(2d)$ implies an even number of $-1$ and $1$. Consequently, $\mu_{k,2j-1} = \overline{\mu}_{k,2j}$ for all $j = 1, \cdots, d$ and $k = 1, \cdots, n$. This leads to $\omega_{2j-1}^{(k)} = -\omega_{2j}^{(k)}$, and therefore $\overline{w}_{2j-1} = -\overline{w}_{2j}$. For any $w \in (\Gamma^* + \overline{w}_{2j-1})/\mathbb{Z}^n$, we have $-w \in (\Gamma^* + \overline{w}_{2j})/\mathbb{Z}^n$, and then $\lambda_{2j-1,w} = \lambda_{2j,-w}$ for all $j = 1, \cdots, d$ and $w \in (\Gamma^* + \overline{w}_{2j-1})/\Z^n$. Thus, eigenvalues of connection Laplacian on discrete torus always occur in pairs.
	\end{proof}
	
	\section{Heat kernel of standard Laplacian}\label{sec:heat_kernel}
		\subsection{Heat kernel of real torus}
		The heat equation in the continuous case is given by
		$$
		\Delta f(x) + \partial_t f(x) = 0.
		$$
		The fundamental solution $ K(t, x) = (4\pi t)^{-\frac{n}{2}}e^{-\frac{|x|^2}{4t}} $ is known to solve the heat equation on $\mathbb{R}^n$ with the initial value being $\delta_0(x)$.
		
		The heat kernel $ \mathcal{H}_t = e^{-t\Delta } $ satisfies that for all $ f \in \SSS(\mathbb{R}^n) $, $ \mathcal{H}_t f $ is the solution of the heat equation with the initial value $ f $, where $\SSS$ denotes the Schwartz space. Let $ H_t(x, y) = (4\pi t)^{-\frac{n}{2}}e^{-\frac{|x - y|^2}{4t}} $, it can be represented by
		$$
		\mathcal{H}_t f(x) = \int_{\mathbb{R}^n} H_t(x, y)f(y) \, dy = \int_{\mathbb{R}^n} K(x - y)f(y) \, dy = K * f(x).
		$$
		The term ``heat kernel'' may also refer to $ H_t(x, y) $ directly.
		
		For the real torus $ \mathbb{R}T^n = \mathbb{R}^n / \Gamma $ with $ \Gamma = A\mathbb{Z}^n $, we can ascertain through asymptotic expressions that
		$$
		H^{\mathbb{R}T^n}_t(x, y) := \sum_{v \in \Gamma} H_t(x + v, y)
		$$
		is well-defined on $ \mathbb{R}T^n \times \mathbb{R}T^n $. Hence, it constitutes the heat kernel on the real torus.
				
		\subsection{Heat kernel of discrete torus}
		In the discrete setting, we introduce the $I$-Bessel function:
		$$
		I_x(z) := \sum_{k=0}^\infty \frac{\left(\frac{z}{2}\right)^{x+2k}}{k!(x+k)!} = \frac{1}{\pi}\int_0^\pi e^{z\cos\theta}\cos(x\theta) \, d\theta, \quad x \ge 0,
		$$
		and for $x < 0$, we define $I_x = I_{-x}$.
		
		This function solves the differential equation
		$$
		z^2 \frac{d^2w}{dz^2} + z \frac{dw}{dz} - (z^2 + x^2)w = 0.
		$$
		
		The following asymptotic property illustrates the relationship between the $I$-Bessel function and the heat kernel on $\mathbb{R}^n$. This property will be utilized in later discussions:
		\begin{prop}\cite[1.7.9]{jeffrey2008handbook}
			$e^{-t}I_x(t) \sim \frac{1}{\sqrt{2\pi t}}\left(1 - \frac{4x^2 - 1}{8t}\right).$\label{prop:Iasymptotic}
		\end{prop}
		
		Additionally, the $I$-Bessel function satisfies an important recurrence relation:
		\begin{prop}\cite[p. 79]{watson1995treatise}
			$I_{x+1}(z) + I_{x-1}(z) = 2I_x'(z)$, for all $x \in \mathbb{Z}$.\label{prop:recurrent}
		\end{prop}
				
		By direct calculation or utilizing Proposition~\ref{prop:recurrent}, we can deduce that the fundamental solution on $\mathbb{Z}$ is given by
		$$
		K^{\mathbb{Z}}(t, x) = (-1)^x \sum_{n=|x|}^\infty \frac{(-t)^n}{n!} \binom{2n}{n-x} = e^{-2t}I_x(2t),
		$$
		which satisfies the heat equation
		$$
		\Delta_{\mathbb{Z}}f + \partial_t f = 0,
		$$
		with the initial value being
		$$
		\delta_0(x) = 
		\begin{cases}
			1, & \text{if } x = 0,\\
			0, & \text{if } x \neq 0.
		\end{cases}
		$$
		
		For a cyclic graph $\mathbb{Z}_m$, the fundamental solution is
		$$
		K^{\mathbb{Z}_m}(t, x) = \sum_{k = -\infty}^\infty K^{\mathbb{Z}}(t, x + km) = e^{-2t}\sum_{k = -\infty}^{\infty}I_{x + km}(2t).
		$$
		The convergence of the infinite sum is assured by Proposition~\ref{prop:Iasymptotic}. For a general dimension, we define
		$$
		K^{\mathbb{Z}^n}(t, x) = e^{-2nt}I_{x_1}(2t) \cdots I_{x_n}(2t) =: e^{-2nt}I_{x}(2nt),
		$$
		where $x \in \mathbb{Z}^n$ is a vector. This function solves the heat equation
		$$
		\Delta_{\mathbb{Z}^n}f + \partial_t f = 0.
		$$
		For the discrete torus, the fundamental solution is
		$$
		K^{DT^n}(t, x) = \sum_{v \in \Gamma} e^{-2nt} I_{x + v}(2nt) = \sum_{v \in \Gamma} e^{-2nt} \prod_{i=1}^n I_{x_i + v_i}(2t).
		$$
		
		The heat kernel $\mathcal{H}_t^{DT^n} = e^{-t\Delta_{DT^n}}$ can be expressed as $H_t^{DT^n}(x, y) = K^{DT^n}(t, x - y)$.
		
		\subsection{Theta function}
		The theta function for a lattice is defined as
		$$
		\theta_\Gamma(t) := \sum_{x \in \Gamma} e^{-\pi t |x|^2}, \quad t \in \mathbb{R}_+.
		$$
		
		Recall the well-known Poisson Summation Formula for a lattice:
		\begin{thm}[Poisson Summation Formula]\label{thm:PSF}
			For a lattice $\Gamma = A\mathbb{Z}^n$ and a function $f(x) \in \SSS(\mathbb{R}^n)$, we have 
			$$
			\sum_{x \in \Gamma} f(x) = \frac{1}{\mu(\mathbb{R}^n/\Gamma)}\sum_{x \in \Gamma^*}\widehat{f}(x),
			$$
			where $\widehat{f}$ is the Fourier transform given by
			$$
			\widehat{f}(x) := \FFF [f](x) := \int_{\mathbb{R}^n} f(y)e^{-2\pi i \langle x, y \rangle} \, dy,
			$$
			and $\mu(\mathbb{R}^n/\Gamma) = \mu(\mathbb{R}T^n)$ is the volume of the resulting torus.
		\end{thm}
		
		Applying Theorem~\ref{thm:PSF}, we obtain the Theta Inversion Formula:
		\begin{thm}[Theta Inversion Formula]\label{thm:thetafunction}
			For a lattice $\Gamma = A\mathbb{Z}^n$, the following formula holds for $t > 0$:
			$$
			\theta_\Gamma(t) = t^{-\frac{n}{2}}\frac{1}{\mu(\mathbb{R}T^n)}\theta_{\Gamma^*}(t^{-1}).
			$$
		\end{thm}
			
			The theta function for the real torus is defined as
			$$
			\Theta_{\mathbb{R}T^n}(t) := \operatorname{tr} \mathcal{H}_t^{\mathbb{R}T^n} = \operatorname{tr} e^{-t\Delta_{\mathbb{R}T^n}} = \sum_{w \in \Gamma^*} e^{-4\pi^2|w|^2t}.
			$$
			This is compatible with another definition:
			$$
			\Theta_{\mathbb{R}T^n}(t) = \operatorname{tr} \mathcal{H}_t^{\mathbb{R}T^n} = \int_{\mathbb{R}T^n} H_t(x, x) \, dx = \frac{\mu(\mathbb{R}T^n)}{(4\pi t)^{\frac{d}{2}}} \sum_{w \in \Gamma} e^{-\frac{|w|^2}{4t}}.
			$$
			These expressions are equivalent by Theorem~\ref{thm:thetafunction}.
			
			For discrete tori, we have the following formula:
			
			\begin{thm}\label{thm:DPSF}
				For $f \in C^\infty(\mathbb{R}^n/\mathbb{Z}^n)$, $\Gamma = A\mathbb{Z}^n$ with $A$ being an integer matrix, it holds that
				$$
				\sum_{w \in \Gamma^*/\mathbb{Z}^n} f(w) = \mu(DT^n) \sum_{v \in \Gamma} \widehat{f}(v),
				$$
				where $\mu(DT^n) = \mu(\mathbb{Z}^n/\Gamma)$ is the number of vertices in the resulting discrete torus.
			\end{thm}
			
			\begin{proof}
				Applying the Fourier transformation, we have
				\begin{equation}
					f(w) = \sum_{v \in \mathbb{Z}^n} \widehat{f}(v) e^{2\pi i\langle v, w \rangle}. \label{eq:Fourier}
				\end{equation}
				Due to the orthogonal properties of the eigenfunctions of $DT^n = \mathbb{Z}^n/\Gamma$, for all $w_1, w_2 \in \Gamma^*/\mathbb{Z}^n$:
				$$
				\langle e^{2\pi i\langle v, w_1 \rangle}, e^{2\pi i\langle v, w_2 \rangle} \rangle = \sum_{v \in \mathbb{Z}^n/\Gamma} e^{2\pi i\langle v, w_2 - w_1 \rangle} =
				\begin{cases}
					0, & \text{if } w_1 \neq w_2, \\
					\mu(DT^n), & \text{if } w_1 = w_2.
				\end{cases}
				$$
				By duality between $\mathbb{Z}^n/\Gamma$ and $\Gamma^*/\mathbb{Z}^n$:
				$$
				\sum_{w \in \Gamma^*/\mathbb{Z}^n} e^{2\pi i\langle v, w \rangle} =
				\begin{cases}
					0, & \text{if } v \notin \Gamma, \\
					\mu(DT^n), & \text{if } v \in \Gamma.
				\end{cases}
				$$
				Summing up both sides of equation~\eqref{eq:Fourier} over $\Gamma^*/\mathbb{Z}^n$ and exchanging the order of summation, we get:
				$$
				\sum_{w \in \Gamma^*/\mathbb{Z}^n} f(w) = \sum_{v \in \mathbb{Z}^n} \widehat{f}(v) \sum_{w \in \Gamma^*/\mathbb{Z}^n} e^{2\pi i\langle v, w \rangle} = \mu(DT^n) \sum_{v \in \Gamma} \widehat{f}(v).
				$$
			\end{proof}
			
			The theta function for the discrete torus is defined as
			$$
			\theta_{DT^n}(t) = \operatorname{tr} \mathcal{H}_t^{DT^n} = \operatorname{tr} e^{-t\Delta_{DT^n}} = \sum_{w \in \Gamma^*/\mathbb{Z}^n} e^{-t(\sin^2(w_1\pi) + \cdots + \sin^2(w_n\pi))}.
			$$
			Furthermore, given the known fundamental solution, we also have
			$$
			\theta_{DT^n}(t) = \sum_{x \in DT^n} H_t^{DT^n}(x, x) = \mu(DT^n)K^{DT^n}(t, 0) = \mu(DT^n)\sum_{w \in \Gamma}e^{-2nt}I_w(2nt).
			$$
			
			These two formulas coincide due to Theorem~\ref{thm:DPSF} and the following lemma:

			\begin{lemma}[{\cite[Lemma~7]{karlsson2006heat}}]\label{lem:Fourier}
				For the function
				$$
				f(x) := e^{-t(\sin^2(x_1\pi) + \cdots + \sin^2(x_n\pi))},
				$$
				its Fourier transform satisfies
				$$
				\widehat{f}(y) = e^{-2nt}I_y(2nt), \quad \foa y \in \mathbb{Z}^n.
				$$
			\end{lemma}

	\section{Heat kernel of connection Laplacian}\label{sec:heat_kernel_CL}
		\subsection{Heat kernel of real torus}
			Consider the heat equation for the connection Laplacian on $\mathbb{R}T^n$,
			$$
			L_{\mathbb{R}T^n} V(x) + \partial_t V(x) = 0.
			$$
			Like in \cite{Lin2021ManifoldEB}, the heat kernel of the connection Laplacian can be expressed through the eigensystem as 
			$$
			H_t^{\mathbb{R}T^n}(x, y) = \sum_{\lambda \in \Spec({L_{\R T^n}})} e^{-\lambda t} \cl V_\lambda(x) \otimes \cl V_{\lambda}(y) = \frac{1}{\mu(\mathbb{R}T^n)}\sum_{i=1}^d\sum_{w \in \Gamma^* + \overline{w}_i} e^{-\lambda_{i,w}t} V_{i,w}(x) \otimes V_{i,w}(y),
			$$
			where we recall that $V_{i,w}(x) = \mathcal{E}(x)\mathcal{F}_{i,w}(x)$ is the eigenfunction corresponding to $\lambda_{i,w}$ as shown in the proof of Theorem~\ref{thm:RT}, and $\cl V_\lambda$ is the unit eigenfunction. The heat kernel acts on vector-valued functions as
			$$
			\mathcal{H}_t^{\mathbb{R}T^n} U(x) = \frac{1}{\mu(\mathbb{R}T^n)}\sum_{i=1}^d\sum_{w \in \Gamma^* + \overline{w}_i} e^{-\lambda_{i,w}t} V_{i,w}(x)\int_{\mathbb{R}T^n} h_y^F(V_{i,w}(y), U(y)) \, dy,
			$$
			where $h_y^F$ is the quadratic form on the fiber $F_y \cong \mathbb{R}^d$. $\mathcal{H}_t^{\mathbb{R}T^n}U(x)$ is the solution of the heat equation for the connection Laplacian with the initial value being $U(x) \in C^\infty(\mathbb{R}T^n)$.
			
			The theta function for the connection Laplacian is given by:
			$$
			\begin{aligned}
				\Theta_{\R T^n}^L(t)&=\tr \HH_t^{\R T^n}\\
				&=\sum_{i=1}^d \sum_{w\in \Gamma^*} \frac{1}{\mu(\R T^n)}e^{-\lambda_{i,w}t}\int_{\R T^n} h_x^F({V_{i,w}(x),V_{i,w}}(x))dx\\
				&=\sum_{i=1}^d  \sum_{w\in \Gamma^*}e^{-4\pi^2 |w+\cl w_i|^2t}.
			\end{aligned}
			$$
			
			Recall that we have Theorem~\ref{thm:PSF} for lattices, and the translation property of the Fourier transformation 
			\begin{align}
				\mathcal{F}[f(x + x_0)](\xi) = \mathcal{F}[f(x)](\xi)e^{2\pi i \langle x_0, \xi \rangle}, \label{eq:Trans}
			\end{align}
			from which we deduce
			$$
			\Theta_{\mathbb{R}T^n}^L(t) = \sum_{i=1}^d \frac{\mu(\mathbb{R}T^n)}{(4\pi t)^{\frac{n}{2}}} \sum_{w \in \Gamma} e^{-\frac{|w|^2}{4t}} e^{2\pi i \langle w, \overline{w}_i \rangle} =: \sum_{i=1}^d \Theta^i_{\mathbb{R}T^n}(t).
			$$
			
			However, as the eigenvalue of $L_{\mathbb{R}T^n}$ is real, so is the trace $\Theta_{\mathbb{R}T^n}^L(t)$. This implies that we only need to consider the real part, leading to the following theorem:
			
			\begin{thm}
				The theta function for the connection Laplacian on the real torus is
				$$
				\Theta_{\mathbb{R}T^n}^L(t) = \sum_{i=1}^d \frac{\mu(\mathbb{R}^n/\Gamma)}{(4\pi t)^{\frac{n}{2}}} \sum_{w \in \Gamma} e^{-\frac{|w|^2}{4t}} \cos\left(2\pi \langle w, \overline{w}_i \rangle\right).
				$$
			\end{thm}

		\subsection{Heat kernel of discrete torus}
			The heat equation over the discrete torus $DT^n = \mathbb{Z}^n/\Gamma$ is given by
			$$
			L_{DT^n} V(x) + \partial_t V(x) = 0.
			$$
			
			Similarly, the heat kernel of the connection Laplacian on the discrete torus can be expressed as
			$$
			H_t(x, y) = \sum_{\lambda \in \operatorname{Spec}(L_{DT^n})} e^{-\lambda t} \overline{V}_{\lambda}(x) \otimes \overline{V}_{\lambda}(y),
			$$
			and the corresponding theta function is
			\begin{align}
				\theta^L_{DT^n}(t) = \operatorname{tr} \mathcal{H}_t^{DT^n} = \sum_{x \in DT^n} \operatorname{tr} H_t(x, x) = \sum_{i=1}^d \sum_{w \in (\Gamma^* + \overline{w}_i)/\mathbb{Z}^n} e^{-t \sum_{k=1}^n \sin^2(w_k\pi)}. \label{eq:theta}
			\end{align}
			
			To obtain a more specific expression, we apply Theorem~\ref{thm:DPSF} and Lemma~\ref{lem:Fourier} to the theta function~\eqref{eq:theta}. We have:
			$$
			\sum_{w \in (\Gamma^* + \overline{w}_i)/\mathbb{Z}^n} e^{-t\sum_{k=1}^n \sin^2(w_k\pi)} = \mu(DT^n)\sum_{w \in \Gamma} e^{-2nt}I_w(2nt)e^{2\pi i \langle w, \overline{w}_i \rangle}.
			$$
			Summing up over $i = 1, \cdots, d$, we get:
			$$
			\begin{aligned}
				\theta_{DT^n}^L(t) &= \sum_{i=1}^d \sum_{(w + \overline{w}_i) \in \Gamma^*/\mathbb{Z}^n} e^{-t\sum_{k=1}^n \sin^2(w_k\pi)} \\
				&= \sum_{i=1}^d \mu(DT^n) e^{-2nt} \sum_{w \in \Gamma} I_w(2nt)e^{2\pi i \langle w, \overline{w}_i \rangle} =: \sum_{i=1}^d \theta_{DT^n}^i(t).
			\end{aligned}
			$$

			Similarly, we can also take advantage of the fact that the eigenvalue of $L_{DT^n}$ is a real number to obtain the following theorem:
			\begin{thm}
				The theta function for the connection Laplacian on the discrete torus is
				$$
				\theta_{DT^n}^L(t) = \sum_{i=1}^d\mu(DT^n) \sum_{w\in \Gamma} e^{-2nt}I_w(2nt)\cos(2\pi \tri{w,\cl w_{i}}).
				$$
			\end{thm}
		
	\section{Convergence of discrete tori family}\label{sec:converge}
		\subsection{Convergence of eigenvalues}
		As shown in Section~\ref{sec:Converging}, we can study the asymptotic behavior of the spectrum of the connection Laplacian on $\frac{1}{u} DT^n_u \rightarrow \mathbb{R}T^n$. To simplify notations, we use the same symbol for both the real and discrete torus, adding a subscript or a parameter $u$ to denote the symbol on the discrete tori family. For instance, $\sigma^v_u$ or $\sigma^v(u)$ denote the torsion matrix on the discrete torus $DT_u^n$. We define the converging discrete tori family as follows:
		
		\begin{defi}\label{def:converge}
			A family of discrete tori $DT^n_u = \mathbb{Z}^n/M_u\mathbb{Z}^n$ is said to converge to the real torus $\mathbb{R}T^n = \mathbb{Z}^n/A\mathbb{Z}^n$ if $\frac{M_u}{u} \rightarrow A$, and $\Lambda_k(u) \rightarrow \Lambda_k$ for all $k = 1, \cdots, n$, where $\Lambda_k$ is the eigenvalue matrix of the torsion matrix $\sigma^{v_k}$ in Section~\ref{sec:LambdaRT} and Section~\ref{sec:LambdaDT}.
		\end{defi}
			
		\begin{rem}
			Different from definition of converging discrete tori family for standard Laplacian in \cite{Converge1}, we add an additional condition about the converging of $\Lambda_k$ here. It means that the torsion matrix converges in the sense of eigenvalue. This condition gives the information about converging of $\cl w_j(u)$.
			
		\end{rem}	
		
		\begin{prop}\label{prop:eigenvalueConverge}
			For a converging discrete tori family $\frac{1}{u}DT^n_u \rightarrow \mathbb{R}T^n$, we have 
			$$u^2 \operatorname{Spec}(L_{DT^n_u}) \rightarrow \operatorname{Spec}(L_{\mathbb{R}T^n}),\quad u\rightarrow\infty.$$
			It means that for each eigenvalue $\lambda$ of $L_{\mathbb{R}T^n}$, there exists a corresponding sequence $\{\lambda_u\}$ of eigenvalues from $L_{DT^n_u}$ such that $u^2 \lambda_u \rightarrow \lambda$ as $u \rightarrow \infty$. In the case of eigenvalues with multiplicities, the sequences $\{\lambda_u\}$ corresponding to each instance of the multiplicity are distinct except finite terms.
			
		\end{prop}
		
		\begin{proof}
			For any $\lambda_{j,w} \in \operatorname{Spec}(L_{\mathbb{R}T^n})$, where $w = (A^T)^{-1}z \in \Gamma^*$ for fixed $z \in \mathbb{Z}^n$, let $w(u) = (M_u^T)^{-1}z \in \Gamma^*_u$. Since $\frac{M_u}{u} \rightarrow A$, we have 
			\begin{align}
				w(u) = \frac{(A^T)^{-1}}{u}z + o\left(\frac{1}{u}\right),\quad \overline{w}_j(u) = -(M_u^T)^{-1}\left(\frac{\omega^{(k)}_j}{2\pi}\right)_{k=1}^n = \frac{\overline{w}_j}{u} + o\left(\frac{1}{u}\right).\label{eq:w}
			\end{align}
			Then by Taylor expansion,
			$$
			u^2\lambda_{j,w(u)} = 4u^2\sum_{k=1}^n \sin^2(w_k(u) + \overline{w}_{j,k}(u)\pi) = 4\sum_{k=1}^n (w_k + \overline{w}_{j,k})^2\pi^2 + o(1) = \lambda_{j,w} + o(1).
			$$
			Thus, for each eigenvalue $\lambda_{j,w}$, there exists a corresponding sequence $\{u^2\lambda_{j,w(u)}\}$ that converges to $\lambda_{j,w}$. For eigenvalues with multiplicities, if $j$ is the same, then $w=(A^T)^{-1}z$ differs. Consequently, $w(u)=(M_u^T)^{-1}z \in \Gamma^*_u/\mathbb{Z}^n$ also differs for sufficiently large $u$. This completes the conclusion.
			
		\end{proof}	
	
	 	Proposition~\ref{prop:eigenvalueConverge} proves the first part of Theorem~\ref{thm:converge}. As for the log-determinant, we turn to the convergence of theta function. 
				
		\subsection{Convergence of theta function}
		We begin by presenting a more general version for the convergence of the theta function for the standard Laplacian compared to \cite[Proposition 5.2]{Converge2}, where we allow $M_u$ and $A$ to be non-diagonal matrices.
		
		\begin{lemma}\label{lem:theta}
			The theta function for the standard Laplacian on the discrete torus $DT^n_u$ converges to the theta function of the real torus $\mathbb{R}T^n$ pointwise in the following sense:
			$$
			\theta_{DT^n_u}(u^2t) \rightarrow \Theta_{\mathbb{R}T^n}(t), \quad \text{as } u \rightarrow \infty.
			$$
		\end{lemma}
		
		\begin{proof}
		For the theta function of the discrete torus $DT^n_u = \mathbb{Z}^n/M(u)\mathbb{Z}^n$, we have:
		\begin{align}
			\theta_{DT^n_u}(u^2t) &= \mu(DT^n_u)e^{-2u^2nt}\sum_{v \in M_u\mathbb{Z}^n}\prod_{k=1}^n I_{v_k}\left(2u^2t\right) \nonumber\\
			&= \sum_{v \in M_u\mathbb{Z}^n}\frac{\mu(DT^n_u)}{v_1\cdots v_n}\prod_{k=1}^nv_ke^{-2u^2t}I_{v_k}\left(2u^2t\right) \quad \text{if $v_k \not\rightarrow \infty$, replace $v_k$ by $u$}.\label{eq:summation}
		\end{align}
		For each $v(u) \in M_u\mathbb{Z}^n$, we have $v_k(u) = \sum_{i=1}^n m_{ki}z_i$ for some $z = (z_1, \cdots, z_n) \in \mathbb{Z}^n$. Let $\widetilde{v} = Az$ be the corresponding vector of $v(u)$, so $\widetilde{v}_k = \sum_{i=1}^n a_{ki}z_i$ and $\frac{v_k}{u} \rightarrow \widetilde{v}_k$ as $\frac{m_{ij}}{u} \rightarrow a_{ij}$. As a result, $v_k \rightarrow \infty$ if and only if $\widetilde{v}_k \neq 0$. Suppose $\widetilde{v}_k \neq 0$ for all $k = 1, \cdots, n$, then
		\begin{align}
			\frac{\mu(DT^n_u)}{v_1\cdots v_n} = \frac{|\det M(u)|}{u^n} \cdot \frac{u^n}{v_1\cdots v_n} \rightarrow \frac{\mu(\mathbb{R}T^n)}{\widetilde{v}_1\cdots \widetilde{v}_n}. \label{eq:volume}
		\end{align}
		If $\widetilde{v}_k = 0$ for some $k$, replace $v_k$ by $u$ and $\widetilde{v}_k$ by $1$ in \eqref{eq:volume}.
		
		Next, for $\widetilde{v}_k \neq 0$ and $v_k \rightarrow \infty$, we have:
		$$
		v_k e^{-2u^2t}I_{v_k}(2u^2t) = \frac{1}{\pi}\int_0^{\pi}e^{-2u^2t(1-\cos\theta)}\cos(v_k\theta) d(v_k\theta) = \frac{1}{\pi}\int_0^{v_k\pi} e^{-2u^2t(1-\cos\frac{\theta}{v_k})}\cos\theta d\theta.
		$$
		By Taylor expansion, for some $c > 0$, we have the integrable bound
		$$
		\left|e^{-2u^2t\left(1-\cos\frac{\theta}{v_k}\right)}\cos \theta\right| \le e^{-cu^2t\left(\frac{\theta}{v_k}\right)^2} \le e^{-ct\left(\frac{\theta}{2\widetilde{v}_k}\right)^2}, \quad \theta \in [0, v_k\pi], \quad u \text{ sufficiently large}.
		$$
		Then, by the Lebesgue Dominated Convergence Theorem,
		$$
		v_k e^{-2u^2t}I_{v_k}(2u^2t) \rightarrow \frac{1}{\pi}\int_0^\infty e^{-t\left(\frac{\theta}{\widetilde{v}_k}\right)^2}\cos\theta d\theta = \frac{\widetilde{v}_k}{\sqrt{4\pi t}}e^{-\frac{\widetilde{v}_k^2}{4t}}.
		$$
		
		If $ v_k \not\rightarrow \infty $, we have $ \widetilde{v}_k = 0 $, and $ v_k $ is bounded for all $ u $. Then similarly,
		$$
		ue^{-2u^2t}I_{v_k}(2u^2t) = \frac{1}{\pi}\int_0^{u\pi} e^{-2u^2t\left(1 - \cos\frac{\theta}{u}\right)}\cos\left(\frac{v_k}{u}\theta\right)d\theta.
		$$
		This expression also has an integrable bound $ e^{-ct\left(\frac{\theta}{2}\right)^2} $ for $ u $ sufficiently large. Therefore, as $ u \rightarrow \infty $, it converges to
		$$
		\frac{1}{\pi}\int_0^{\infty} e^{-t\theta^2}d\theta = \frac{1}{\sqrt{4\pi t}}.
		$$
		
		Thus, in any case, for all $ z \in \mathbb{Z}^n $, let $ v(u) = M_uz $ and $ \widetilde{v} = Az $. We have
		$$
		\mu(DT^n_u)\prod_{k=1}^n e^{-2u^2t}I_{v_k(u)}(2u^2t) \rightarrow \mu(\mathbb{R}T^n)\prod_{k=1}^n \frac{e^{-\frac{\widetilde{v}_k^2}{4t}}}{\sqrt{4\pi t}}.
		$$
		
		Following the same technique as in \cite[Proposition~5.2]{Converge2}, we can provide summable bounds that are uniformly applicable for $u$ sufficiently large. This part is explained in detail in the Appendix~\ref{app:sum}. Thus, by interchanging the limit in $u$ and the summation in \eqref{eq:summation}, we have
		$$
		\theta_{DT^n_u}(u^2t) \rightarrow \sum_{\widetilde{v} \in A\mathbb{Z}^n} \mu(\mathbb{R}T^n) \prod_{k=1}^n \frac{e^{-\frac{\widetilde{v}_k^2}{4t}}}{\sqrt{4\pi t}} = \mu(\mathbb{R}T^n) \sum_{\widetilde{v} \in A\mathbb{Z}^n} \frac{1}{(4\pi t)^{\frac{n}{2}}} e^{-\frac{|\widetilde{v}|^2}{4t}} = \Theta_{\mathbb{R}T^n}(t).
		$$			
		\end{proof}
		
		\begin{lemma}
			The theta function for the connection Laplacian on the discrete torus $DT^n_u$ converges to the theta function of the real torus $\mathbb{R}T^n$ in the following sense:
			$$
			\theta_{DT^n_u}^L(u^2t) \rightarrow \Theta_{\mathbb{R}T^n}^L(t), \quad \text{as } u \rightarrow \infty.
			$$
		\end{lemma}
		
		\begin{proof}
			Recall that we have 
			$$\theta_{DT^n_u}^L(u^2t)=\sum_{i=1}^d \theta_{DT^n_u}^i(u^2t),\quad \Theta_{\R T^n}(t)=\sum_{i=1}^d \theta_{\R T^n}^i(t)$$
			in Section~\ref{sec:heat_kernel_CL}, we only need to demonstrate convergence for each $i = 1, \cdots, d$. The proof is a modification of the proof for Lemma~\ref{lem:theta}. On the discrete torus $DT^n_u$, for $v(u) = M_uz$ and $\widetilde{v} = Az$, where $z \in \mathbb{Z}^n$, we have
			$$
			e^{2\pi i \langle v, \cl w_{j}(u) \rangle} = \prod_{k=1}^n e^{-i z_k \omega_{j}^{(k)}(u)} \rightarrow \prod_{k=1}^n e^{-i z_k \omega_{j}^{(k)}} = e^{2\pi i \langle \widetilde{v}, \cl {w}_{j} \rangle}.
			$$
			
			For $v_k \rightarrow \infty$, we have
			$$
			v_k e^{-2u^2t}I_{v_k}(2u^2t)e^{2\pi i \langle v, \cl w_{j}(u) \rangle} \rightarrow \frac{\widetilde{v}_k}{\sqrt{4\pi t}} e^{-\frac{\widetilde{v}_k^2}{4t}}e^{2\pi i \langle \widetilde{v}, \cl {w}_{j} \rangle},
			$$
			For $v_k \not\rightarrow \infty$, we have $\widetilde{v}_k = 0$, and the convergence follows similarly by replacing $v_k$ with $u$.
			
			Because $|e^{2\pi i\tri{v,\cl w_j}}|\equiv 1,$ we still have summable bounds as in Lemma~\ref{lem:theta}. By interchanging the limit and the summation, for each $j = 1, \cdots, d$, we have
			$$
			\theta_{DT^n_u}^j(u^2t) \rightarrow \sum_{\widetilde{v} \in A\mathbb{Z}^n} \mu(\mathbb{R}T^n) \frac{1}{(4\pi t)^{\frac{n}{2}}} e^{-\frac{|\widetilde{v}|^2}{4t}} e^{2\pi i \langle \widetilde{v}, \cl{w}_{j} \rangle} = \theta_{\R T^n}^j(t).
			$$
			Therefore, the convergence of the entire theta function is established.
		\end{proof}
		
		\subsection{Convergence of the log-determinant}\label{sec:main}
		
		The regularized log determinant of the standard Laplacian is defined through the derivative of the spectral zeta function as outlined in \cite{Converge1,Converge2}. In a similar way, we define the regularized log determinant of the connection Laplacian as follows.
		
		\begin{defi}
			The spectral zeta function of the connection Laplacian is defined to be
			$$
			\zeta(s) := \sum_{\lambda \in \operatorname{Spec}(L) \setminus \{0\}} \frac{1}{\lambda^s}.
			$$
		\end{defi}
			
		The spectral zeta function of the connection Laplacian on the real torus is then
		$$
		\zeta(s) = \sum_{i=1}^d \sum_{\substack{w \in \Gamma^* + \overline{w}_i \\ w \neq 0}} (2\pi|w|)^{-2s} =: \sum_{i=1}^d \zeta_i(s).
		$$
		
		For $ i = 1, \cdots, d $, the case $ w = 0 $ is only possible when $ \overline{w}_i = 0 $. Thus, we will always consider the situation that $ \overline{w}_i \neq 0 $ in later discussions, as the case when $ \overline{w}_i = 0 $ degenerates to the standard Laplacian, which has already been detailed in \cite{Converge2}.
			
		For each $\zeta_i(s)$, it converges when $\Re(s) > \frac{n}{2}$. By applying the Mellin transform, we can demonstrate that it extends meromorphically to $\mathbb{C}$, and $0$ is a holomorphic point as per \cite{Berline1992HeatKA}. Formally, since $\left. \frac{d}{ds}\frac{1}{\lambda^s} \right|_{s=0} = -\log \lambda$, it is reasonable to adopt the following definition.
		
		\begin{defi}[Ray-Singer \cite{RAY1971145}]
			The regularized log-determinant $\log \det{}^{*}L$ is defined by
			$$
			\log \det{}^{\ast}L = -\zeta^{\prime}(0) = -\sum_{i=1}^d \zeta_i^{\prime}(0).
			$$
		\end{defi}
		
		\begin{prop}
			For all $i = 1, \cdots, d$, $s = 0$ is a holomorphic point of $\zeta_i(s)$, and if $\overline{w}_i \neq 0$, then
			$$
			\zeta_i^{\prime}(0) = \int_1^\infty \Theta_{\mathbb{R}T^n}^i(t)\frac{dt}{t} + \int_0^1 \left(\Theta_{\mathbb{R}T^n}^i(t) - \frac{\mu(\mathbb{R}T^n)}{\sqrt{4\pi t}^n}\right)\frac{dt}{t} - \frac{2\mu(\mathbb{R}T^n)}{n\sqrt{4\pi}^n}.
			$$
		\end{prop}
		
		\begin{proof}
			Using the Mellin transform and the Monotone Convergence Theorem, we have:
			$$
			\begin{aligned}
				\zeta_i(s) &= \sum_{w \in \Gamma^* + \overline{w}_i} \frac{1}{\Gamma(s)}\int_0^\infty e^{-4\pi^2 |w|^2t}t^{s-1}dt \\
				&= \frac{1}{\Gamma(s)}\int_0^\infty \sum_{w \in \Gamma^* + \overline{w}_i}e^{-4\pi^2|w|^2t}t^{s-1}dt \\
				&= \frac{1}{\Gamma(s)}\int_0^\infty \Theta_{\mathbb{R}T^n}^i(t)t^{s-1}dt \\
				&= M[\Theta_{\mathbb{R}T^n}^i(t)].
			\end{aligned}
			$$			
			
			As $\Theta_{\mathbb{R}T^n}^i(t) = \sum_{w \in \Gamma^* + \overline{w}_i} e^{-4\pi^2|w|^2t}$, we know that there exists $C > 0$ such that
			$$
			\Theta_{\mathbb{R}T^n}^i(t) = O(e^{-Ct}), \quad \text{for } t \ge 1.
			$$
			Furthermore, we have $\Theta_{\mathbb{R}T^n}^i(t) = \frac{\mu(\mathbb{R}T^n)}{\sqrt{4\pi t}^n}\sum_{w \in \Gamma^*}e^{-\frac{|w|^2}{4t}}e^{2\pi i \langle w, \overline{w}_i \rangle}$, and then
			$$
			\Theta_{\mathbb{R}T^n}^i(t) = \frac{\mu(\mathbb{R}T^n)}{\sqrt{4\pi t}^n} + O(e^{-\frac{C}{t}}), \quad \text{for } 0 < t \le 1.
			$$
			
			Now we claim that the extension can be written as
			$$
			\begin{aligned}
				\zeta_i(s)&=\frac{1}{\Gamma(s)}\left(\int_1^\infty \Theta_{\R T^n}^i(t)t^{s-1}dt+\int_0^1 \left(\Theta_{\R T^n}^i(t)-\frac{\mu(\R T^n)}{\sqrt{4\pi t}^n}\right)t^{s-1}dt+\int_0^1 \frac{\mu(\R T^n)}{\sqrt{4\pi t}^n}t^{s-1}dt\right) \\
				&= \frac{1}{\Gamma(s)}\int_1^\infty \Theta_{\R T^n}^i(t)t^{s-1}dt+\frac{1}{\Gamma(s)}\int_0^1 \left(\Theta_{\R T^n}^i(t)-\frac{\mu(\R T^n)}{\sqrt{4\pi t}^n}\right)t^{s-1}dt+\frac{\mu(\R T^n)}{(s-\frac{n}{2})\sqrt{4\pi}^n\Gamma(s)}
			\end{aligned}
			$$
			
			As $\frac{1}{\Gamma(s)}=s+O(s^2)$ when $s\rightarrow 0,$ $s=0$ is a holomorphic point of the above expression, and we have
			$$
			\zeta_i(0)=0, \quad \zeta_i'(0)=\int_1^\infty \Theta_{\R T^n}^i(t)\frac{dt}{t}+\int_0^1 \left(\Theta_{\R T^n}^i(t)-\frac{\mu(\R T^n)}{\sqrt{4\pi t}^n}\right)\frac{dt}{t}-\frac{2\mu(\R T^n)}{n\sqrt{4\pi}^n}.
			$$
			
		\end{proof}
	
		Similarly, for the connection Laplacian on the discrete torus, we have $\zeta_i(s) = M[\theta_{DT^n}^i(t)]$ for $\Re(s) > \frac{n}{2}$. However, we need to use the Gauss transformation to obtain the extension near $s = 0$.
		
		\begin{prop}\label{prop:DT-log-determinant}
			For the discrete torus $DT^n = \mathbb{Z}^n / A\mathbb{Z}^n$, and for $i = 1, \cdots, d$, if $\overline{w}_i \neq 0$, we have
			$$
			\sum_{w \in (\Gamma^* + \overline{w}_i) / \mathbb{Z}^n} \log(\lambda_w) = \mu(DT^n)\mathcal{I}_n(0) + \mathcal{H}^i(0), \quad \lambda_w = 4 \sum_{k=1}^n \sin^2(w_k\pi),
			$$
			where
			$$
			\mathcal{I}_n(s) = -\int_0^\infty (e^{-s^2t}e^{-2nt}I_0(2t)^n - e^{-t})\frac{dt}{t}
			$$
			only depends on $n$, and
			$$
			\mathcal{H}^i(s) = -\int_0^\infty e^{-s^2t}(\theta_{DT^n}^i(t) - \mu(DT^n)e^{-2nt}I_0(2t)^n)\frac{dt}{t}.
			$$
		\end{prop}
		
		\begin{proof}
			By the definition of $\theta_{DT^n}^i(t)$, calculating $2s \int_0^\infty e^{-s^2t} \theta_{DT^n}^i(t) \, dt$ gives us
			\begin{align}
				\sum_{w \in (\Gamma^* + \overline{w}_i) / \mathbb{Z}^n} \frac{2s}{s^2 + \lambda_w} &= \mu(DT^n) \cdot 2s \int_0^\infty e^{-s^2t} e^{-2nt} I_0(2t)^n \, dt \label{eq:ODE} \\
				&\quad + 2s \int_0^\infty e^{-s^2t} \left(\theta_{DT^n}^i(t) - \mu(DT^n) e^{-2nt} I_0(2t)^n\right) \, dt \nonumber
			\end{align}
			for all $s \in \mathbb{C}$ with $\Re(s^2) > 0$. We separate the integral into two parts for the convenience of later discussion.
			
			Now, we attempt to find the primitive function of both sides. The left-hand side of equation \eqref{eq:ODE} is the derivative of
			$$
			f(s) = \sum_{w \in (\Gamma^* + \overline{w}_i) / \mathbb{Z}^n} \log(s^2 + \lambda_w),
			$$
			with the asymptotic condition $f(s) = \mu(DT^n)\log(s^2) + o(1)$ as $s \rightarrow \infty$. For the right-hand side of equation \eqref{eq:ODE}, the primitive function can be written as
			$$
			-\mu(DT^n)\int_0^\infty \left(e^{-s^2t}e^{-2nt}I_0(2t)^n - e^{-t}\right)\frac{dt}{t} - \int_0^\infty e^{-s^2t}\left(\theta_{DT^n}^i(t) - \mu(DT^n)e^{-2nt}I_0(2t)^n\right)\frac{dt}{t}.
			$$
			
			By the definition in the proposition, it is
			$$
			\mu(DT^n)\II_n(s)+\HH^i(s).
			$$
			They have the asymptotic condition
			$$
			\II_n(s)=\log(s^2)+o(1),\quad \HH^i(s)=o(1),\quad \text{as }s\rightarrow\infty.
			$$
			
			Thus, we have
			$$
			f(s)=\sum_{w\in (\Gamma^*+\cl w_i)/\Z^n}\log(s^2+\lambda_w)=\mu(DT^n)\II_n(s)+\HH^i(s),\quad \Re(s^2)>0.
			$$
			However, both sides are well defined and holomorphic near $s=0,$ so
			$$
			\sum_{w\in (\Gamma^*+\cl w_i)/\Z^n}\log(\lambda_w)=\mu(DT^n)\II_n(0)+\HH^i(0).
			$$
		\end{proof}
		
		\begin{proof}[Proof of Theorem~\ref{thm:converge}]
			Suppose $\overline{w}_i \neq 0$, and we have $\cl w_i(u)\neq 0$ except for finitely many terms by \eqref{eq:w}. As $\theta_{DT^n_u}^i(u^2t) \rightarrow \Theta_{\mathbb{R}T^n}^i(t)$, we have
			$$
			\int_1^\infty \left(\theta_{DT^n_u}^i(u^2t) - \mu(DT^n_u)e^{-2u^2nt}I_0(2u^2t)^n\right)\frac{dt}{t} \rightarrow \int_1^\infty \Theta_{\mathbb{R}T^n}^i(t)\frac{dt}{t} - \frac{2\mu(\mathbb{R}T^n)}{n\sqrt{4\pi}^n}
			$$
			by the Lebesgue Dominated Convergence Theorem, where the upper bound is established in \cite[Section~5]{Converge2} and modified in the Appendix~\ref{app:int}. Similarly,
			$$
			\int_0^1 \left(\theta_{DT^n_u}^i(u^2t) - \mu(DT^n_u)e^{-2u^2nt}I_0(2u^2t)^n\right)\frac{dt}{t} \rightarrow \int_0^1 \left(\Theta_{\mathbb{R}T^n}^i(t) - \frac{\mu(\mathbb{R}T^n)}{\sqrt{4\pi t}^n}\right)\frac{dt}{t}.
			$$
			Therefore, $\mathcal{H}_u^i(0) \rightarrow -\zeta_i'(0)$, and then
			\begin{align}
				\sum_{w \in (\Gamma^*_u + \overline{w}_i(u)) / \mathbb{Z}^n} \log(\lambda_w) = \mu(DT^n_u)\mathcal{I}_n(0) - \zeta_i'(0) + o(1).\label{eq:case1}
			\end{align}
			
			For the case when $\overline{w}_i = 0$, if $\overline{w}_i(u) \equiv 0$ except for finitely many terms, we already have
			\begin{align}
				\sum_{w \in (\Gamma^*_u \setminus \{0\}) / \mathbb{Z}^n} \log(\lambda_w) = \mu(DT^n_u)\mathcal{I}_n(0) - \zeta_i'(0) + \log u^2 + o(1)\label{eq:case2}
			\end{align}
			by the work in \cite{Converge2}.
			
			The only remaining case is when $\overline{w}_i = 0$, but $\overline{w}_i(u) \neq 0$ for infinitely many terms. In this case, we have $\cl w_i(u)=o\left(u^{-1}\right).$ We need to slightly modify $\log \det{}^* (L_{DT^n_u})$ for this situation, that is to define
			\begin{align}
				\log \det{}^*(L_{DT^n_u}):=\sum_{i=1}^d \sum_{w\in (\Gamma^*_u\setminus \{0\}+\cl w_i(u))/\Z^n} \log(\lambda_w) +\sum_{i=1}^d \chi(\cl w_i(u)) \log(\lambda_{\cl w_i(u)}),\label{eq:logdet}
			\end{align}
			where $\chi(\cl w_i(u))=0$ if $\cl w_i(u)=o(u^{-1}),$ and $1$ otherwise. This case is left to the Appendix~\ref{app:B}.
			
			In summary, for $i=1,\cdots,d,$ by taking \eqref{eq:case1} into \eqref{eq:logdet} when $\cl w_i\neq 0$, and taking \eqref{eq:case3} in Appendix~\ref{app:B} into \eqref{eq:logdet} when $\cl w_i=0$, we have
			$$
			\log \det{}^* (L_{DT^n_u}) = \mu(DT^n_u)\mathcal{I}_n(0)d + (\dim \ker L_{\mathbb{R}T^n}) \log u^2 + \log \det{}^* (L_{\mathbb{R}T^n}) + o(1).
			$$
			
		\end{proof}

\appendix
\section{Upper bound for theta function}
	\subsection{Summable bound on lattice}\label{app:sum}
	In the proof of Lemma~\ref{lem:theta}, we need to provide a uniform upper bound in $u$ for 
	$$
	\theta_{DT^n_u}(u^2t) = \mu(DT^n_u)\sum_{v(u) \in M_u\mathbb{Z}^n}\prod_{k=1}^n e^{-2u^2t}I_{v_k(u)}(2u^2t)
	$$
	to interchange the limit in $u$ and summation, thereby proving the convergence of the theta function.
	
	We first estimate the lower bound of $\frac{|v_k(u)|}{u}.$ Every $v(u) \in M_u\mathbb{Z}^n$ corresponds to a vector $z \in \mathbb{Z}^n$, yielding $v_k(u) = (M_uz)_k = \langle m_k(u), z \rangle$, where $m_k$ denotes the $k$'th row vector of $M_u$. Let $\alpha_k$ represent the $k$'th row vector of $A$, then we have $\widetilde{v}_k = (Az)_k = \langle \alpha_k, z \rangle$. As $\frac{M_u}{u} \rightarrow A$, it follows that $\left\|\frac{m_k(u)}{u}\right\|_2 \ge \frac{1}{2}\left\|\alpha_k\right\|_2$ for sufficiently large $u$. Consequently,
	$$
	\frac{|v_k(u)|}{u} = \left\|\frac{m_k(u)}{u}\right\|_2 |\cos\theta(m_k(u), z)| \|z\|_2 \ge \frac{1}{2}\|\alpha_k\|_2 |\cos\theta(m_k(u), z)| \|z\|_\infty,
	$$
	where $\theta(m_k(u), z)$ denotes the angle between $m_k(u)$ and $z$.
	
	Given that $\{\alpha_k\}_{k=1}^n$ is non-degenerate, we have $C_0 := \min_{z \in S^n} \max_k |\cos\theta(\alpha_k, z)| > 0$. Moreover, for all $z \in \mathbb{Z}^n \setminus \{0\}$ and sufficiently large $u$, it can be shown that
	$$\max_k|\cos\theta(m_k(u), z)| = \max_k\left|\cos\theta\left(\frac{m_k(u)}{u}, \frac{z}{\|z\|_2}\right)\right| \ge \frac{C_0}{2}$$
	by the convergence of $\frac{M_u}{u} \rightarrow A$.
	
	Let $k_0(u,z)=\arg\max_k |\cos\theta(m_k(u),z)|,$ there exists $u_0$ large enough, such that
	$$\frac{|v_{k_0}(u)|}{u}\ge \frac{\min\|\alpha_k\|_2}{4}\|z\|_\infty=:2C\|z\|_\infty,\quad \foa z\neq 0,\quad u>u_0.$$
		
	Now we can give the required upper bound. By \cite[Corollary~4.4]{Converge2}, we have
	$$
	ue^{-2u^2t}I_{v_k(u)}(2u^2t) \le \frac{1}{\sqrt{2t}}\left(1 + \frac{|v_k(u)|}{2u^2t}\right)^{-\frac{|v_k(u)|}{2}} \le 1.
	$$	
	
	By the fact that $g(c,y) = \left(1 + \frac{c}{y}\right)^{-y}$ monotonically decreases with respect to both $y$ and $c$ in $y,c>0,$ it follows
	$$
	\begin{aligned}
		ue^{-2u^2t}I_{v_{k_0}(u)}(2u^2t) &\le \frac{1}{\sqrt{2t}}\left(1 + \frac{|v_{k_0}(u)|}{2u^2t}\right)^{-\frac{|v_{k_0}(u)|}{2}}\\
		&= \frac{1}{\sqrt{2t}}g^{\frac{|v_{k_0}(u)|^2}{4u^2t}}\left(\frac{|v_{k_0}(u)|}{2}\right)\\
		&\le \frac{1}{\sqrt{2t}}g^{\frac{C^2\|z\|_\infty^2}{t}}\left(Cu_0\|z\|_\infty\right)\\
		&=\frac{1}{\sqrt{2t}}\left(1+\frac{C\|z\|_\infty}{u_0t}\right)^{-Cu_0\|z\|_\infty},&\foa u\ge u_0.
	\end{aligned}
	$$
	
	Additionally assume that $\frac{\mu(DT^n_u)}{u_n}\le 2\mu(\R T^n)$ for $u\ge u_0,$ we obtain
	$$
	\begin{aligned}
		\mu(DT^n_u)\sum_{v(u)\in M_u\Z^n}\prod_{k=1}^ne^{-2u^2t}I_{v_k(u)}(2u^2t)&\le\frac{\mu(DT^n_u)}{u^n} \frac{1}{\sqrt{2t}^{n-1}}\sum_{v(u)\in M_u\Z^n} ue^{-2u^2t}I_{v_{k_0}(u)}(2u^2t) \\
		&\le \frac{2\mu(\R T^n)}{\sqrt{2t}^n}\sum_{m=0}^\infty \sum_{\|z\|_\infty=m} \left(1+\frac{Cm}{u_0t}\right)^{-Cu_0m}\\
		&\le \frac{2\mu(\R T^n)}{\sqrt{2t}^n}\sum_{m=0}^\infty 2n(2m+1)^{n-1}\left(1+\frac{C}{u_0t}\right)^{-Cu_0m}.
	\end{aligned}
	$$ 
	For every fixed $t > 0$, this provides a uniform bound for $u \geq u_0$. It leads to the feasibility of interchanging the limit and sum in the proof of Lemma~\ref{lem:theta}.
	
	\subsection{Integrable bound for $t$ in $\R$}\label{app:int}
	In the proof of Theorem~\ref{thm:converge} in Section~\ref{sec:converge}, we need to establish a control function for $\theta^i_{DT^n_u}(u^2t)$ under the integral in $t$ over $\R$. For the integral over the range from 1 to infinity, we employ a similar approach to that in \cite[Lemma~5.3]{Converge2}. We express $\theta_{DT^n_u}^i(u^2t)$ as follows:
	$$
	\begin{aligned}
		\theta_{DT^n_u}^i(u^2t)&=\sum_{w\in (\Gamma^*+\cl w_i)_u/\Z^n} e^{-4u^2t \sum_{k=1}^n\sin^2(w_k\pi)}\\
		&=\sum_{w\in (\Gamma^*+\cl w_i)_u/\Z^n} \prod_{k=1}^ne^{-4 u^2 \sin^2 (\min\{w_k,1-w_k\}\pi)t}.
	\end{aligned}
	$$
	
	For each $ w(u) \in (\Gamma^*_u + \overline{w}_i(u)) / \mathbb{Z}^n $, $ w(u) $ is represented as $ (M_u^{-1})^T(z + z_0) $ for some $ z \in \mathbb{Z}^n $, where $ z_0 = -\left(\frac{\omega_j^{(k)}}{2\pi}\right)_{k=1}^n\neq 0 $ such that $ (M_u^{-1})^Tz_0 = \overline{w}_{i}(u) $. We always associate $w$ with the corresponding $z.$ If $ w_k(u) \in [0, \frac{1}{2}] $, then by the inequality $ \sin x \ge \frac{2}{\pi}x $ for $ x \in [0, \frac{\pi}{2}] $, we have
	$$
	\begin{aligned}
		u\sin(w_k(u)\pi)&\ge 2uw_k(u)\\
		&=2u((M_u^{-1})^T (z+z_0))_k\\
		&=2u|\tri{\gamma_k(u),z+z_0}|,
	\end{aligned}
	$$
	where $\gamma_k(u)$ is the $k$'th column vector of $M_u^{-1}$. Let $\beta_k$ be the $k$'th column vector of $A^{-1},$ then $u\gamma_k(u)\rightarrow \beta_k$ due to $uM_u^{-1}\rightarrow A^{-1}.$ As a result, for sufficiently large $u,$ we have $2u\|\gamma_k(u)\|_2\ge \|\beta_k\|_2$ for all $k=1,\cdots,n.$
	
	If $w_k(u)\in (\frac{1}{2},1],$ similarly we have
	$$
	\begin{aligned}
		u\sin((1-w_k(u))\pi)&\ge 2u(1-w_k(u))\\
		&=2u((M_u^{-1})^T(m_k(u)-z-z_0))_k\\
		&=2u|\tri{\gamma_k(u),z+z_0-m_k(u)}|,
	\end{aligned}
	$$
	where $m_k(u)$ is the $k$'th column vector of $M_u.$ We have $\tri{m_i(u),\gamma_j(u)}=\delta_{ij},$ where $\delta_{ij}$ is the Kronecker symbol.
	
	For any fixed $x=z+z_0,$ assume $w_{k_i}(u)\in (\frac{1}{2},1]$ for $i=1,\cdots,r,$ and $w_j(u)\in [0,\frac{1}{2}]$ for the remaining $j\notin \{k_i\}_{i=1}^r.$ Notice that $\tri{\gamma_k(u),x}=\tri{\gamma_k(u),x-m_j(u)}$ for all $j\neq k,$ we turn to consider $x'=x-\sum_{i=1}^r m_{k_i}(u).$ Then for all $k=1,\cdots,n,$
	$$
	\begin{aligned}
		u\sin (\min\{w_k,1-w_k\}\pi)&\ge 2u|\tri{\gamma_k(u),x'}|\\
		&\ge 2u\|\gamma_k(u)\|_2 |\cos\theta(\gamma_k(u),x')| \|x'\|_2\\
		&\ge \|\beta_k\|_2 |\cos\theta(\gamma_k(u),x')| \|x'\|_\infty\\
		&\ge \|\beta_k\|_2 |\cos\theta(\gamma_k(u),x')| \left|\left\|z-\sum_{i=1}^r m_{k_i}(u)\right\|_\infty-\|z_0\|_\infty\right|.
	\end{aligned}
	$$ 
	
	Similar to Appendix~\ref{app:sum}, there exist $C_0,u_0>0$ such that $\max_k |\cos \theta(\gamma_k(u),x)|\ge C_0$ for all $x\neq 0$ and $u>u_0,$ due to the convergence of $u\gamma_k(u)\rightarrow \beta_k$ and the non-singularity of $\{\beta_k\}_{k=1}^n.$
	
	Notice that for any $\wt z\in \Z^n$ and fixed $u,$ $\wt z=z-\sum_{i=1}^r m_{k_i}(u)$ for some $z\in \Z^n$ only happens for at most $2^n$ times. By setting $c=4\min_k\|\beta_k\|_2 C_0$ and denoting $k_0(u,x) = \arg\max_k |\cos\theta(\gamma_k(u),x')|,$ we obtain
	$$
	\begin{aligned}
		\theta^i_{DT^n_u}(u^2t)\le & \sum_{w\in (\Gamma^* +\cl w_i)_u/\Z^n} \left( e^{-c\left(\|z-\sum_{i=1}^r m_{k_i}(u)\|_\infty -\|z_0\|_\infty\right)^2t}\right) \\
		\le& 2^n\left(e^{-c\|z_0\|_\infty^2 t}+\sum_{m=1}^\infty 2n(2m+1)^{n-1} e^{-c(m-\|z_0\|_\infty)^2t}\right)\\
		= & o(e^{-c't}),\quad c'>0,
	\end{aligned}
	$$
	which is an integrable function.
	
	For the integral over the range from $0$ to $1,$ we need to estimate 
	\begin{align}
		\theta_{DT^n_u}^i(u^2t)-\mu(DT^n_u)(e^{-2u^2t}I_0(2u^2t))^n=\mu(DT^n_u)\sum_{v(u)\in M_u\Z^n\setminus\{0\}}\prod_{k=1}^n e^{-2u^2t}I_{v_k}(2u^2t).\label{eq:aptheta}
	\end{align}
	This can be controlled using the formula derived in Appendix~\ref{app:sum}:
	$$
	\begin{aligned}
		\eqref{eq:aptheta}&\le  \frac{2\mu(\R T^n)}{\sqrt{2t}^n}\sum_{m=1}^\infty 2n(2m+1)^{n-1}\left(1+\frac{C}{u_0t}\right)^{-Cu_0m} \\
		&\le \frac{2\mu(\R T^n)}{\sqrt{2t}^n}\left(\frac{u_0t}{C+u_0t}\right)^{Cu_0}\sum_{m=0}^\infty 2n(2m+3)^{n-1}\left(1+\frac{C}{u_0}\right)^{- Cu_0m}\\
		&\le \wt C(u_0)\mu(\R T^n)\frac{t^{Cu_0}}{\sqrt{2t}^n}.
	\end{aligned}
	$$
	It suffices to select $u_0>\frac{n}{2C}$ such that the term $ \frac{t^{Cu_0}}{\sqrt{2t}^n} $ is integrable on the interval $(0,1)$ with respect to the measure $\frac{dt}{t}$. Thus $\theta^i_{DT^n_u}(u^2t)$ has an integrable controlling function over the entire $\R$ for all $u$ sufficiently large.

\section{Discrete tori with degenerating torsion}\label{app:B}
	To show the converging results of log-determinant for discrete tori with degeneration torsion, we need to eliminate the eigenvalues $\lambda_{\cl w_i(u)}$ if $\cl w_i(u)=o(u^{-1})$ in the formula of log-determinant on discrete tori, as shown in \eqref{eq:logdet}, although $\lambda_{\cl w_i(u)}$ may not be zero.
	
	Properties for the real torus without torsion are already demonstrated in \cite{Converge2}. For discrete tori with torsion, we only need to slightly modify the proof. For example, we have the following proposition for this case:
	
	\begin{prop}
		For the discrete torus $DT^n = \mathbb{Z}^n / A\mathbb{Z}^n$, the following relation holds for the eigenvalues of connection Laplacian:
		\begin{align}
			\sum_{w \in (\Gamma^*\setminus\{0\} + \overline{w}_i) / \mathbb{Z}^n} \log(\lambda_w) = \mu(DT^n)\mathcal{I}_n(0) + \mathcal{H}^i(0), \quad \lambda_w = 4 \sum_{k=1}^n \sin^2(w_k\pi),\label{eq:DT-log-determinant}
		\end{align}
		where
		$$
		\mathcal{I}_n(s) = -\int_0^\infty (e^{-s^2t}e^{-2nt}I_0(2t)^n - e^{-t})\frac{dt}{t}
		$$
		depends only on $n$, and
		$$
		\mathcal{H}^i(s) = -\int_0^\infty (e^{-s^2t}(\theta_{DT^n}^i(t) - \mu(DT^n)e^{-2nt}I_0(2t)^n-1)+e^{-t})\frac{dt}{t}.
		$$
	\end{prop}
	
	\begin{proof}
		We omit the details here, as the proof follows the same approach as \cite[Proposition~3.4]{Converge2}. The primary distinction from Proposition~\ref{prop:DT-log-determinant} lies in the exclusion of an eigenvalue on the left-hand side of \eqref{eq:DT-log-determinant}, affecting the asymptotic behavior for $\HH^i(s)$.
		
	\end{proof}
	
	The remaining process is basically same as in Section~\ref{sec:main} and \cite{Converge2}. Finally we can derive
	\begin{align}
		\sum_{w \in (\Gamma^*_u \setminus \{0\}+\cl w_i(u)) / \mathbb{Z}^n} \log(\lambda_w) = \mu(DT^n_u)\mathcal{I}_n(0) - \zeta_i'(0) + \log u^2 + o(1)\label{eq:case3}
	\end{align}
	when $\cl w_i(u)=o(u^{-1}),$ which is equivalent to $\cl w_i=0.$ This formula is same to the case where $\cl w_i(u)\equiv 0$ as in \eqref{eq:case2}. Take \eqref{eq:case3} back to the proof of Theorem~\ref{thm:converge} in Section~\ref{sec:main}, and then we come to the conclusion.
	
	\textbf{Acknowledgements.} We would like to thank Professor Nicolai Reshetikhin of YMSC for offering such a fascinating topic. Professor Yong Lin is supported by NSFC, no.12071245.
	
\bibliographystyle{plain}
\bibliography{CL}

\end{document}